\documentclass[11pt]{amsart}
\usepackage{amsmath,amssymb}
\usepackage{latexsym}
\usepackage{dsfont}
\usepackage[normalem]{ulem}
\usepackage{epsfig}
\usepackage{pgf}
\usepackage{graphicx,color}
\usepackage{pgf,pgfarrows,pgfnodes,pgfautomata,pgfheaps}
\usepackage{hyperref}
\usepackage[normalem]{ulem}

\usepackage{tikz}

\usepackage[a4paper, total={7.0in, 9.2in}]{geometry}

\usepackage{mathtools}
\mathtoolsset{showonlyrefs}

\usepackage{scalerel}[2014/03/10]
\usepackage[usestackEOL]{stackengine}
\def\dashint{\,\ThisStyle{\ensurestackMath{%
  \stackinset{c}{.2\LMpt}{c}{.5\LMpt}{\SavedStyle-}{\SavedStyle\phantom{\int}}}%
  \setbox0=\hbox{$\SavedStyle\int\,$}\kern-\wd0}\int}
\def\ddashint{\,\ThisStyle{\ensurestackMath{%
  \stackinset{c}{.2\LMpt}{c}{.5\LMpt+.2\LMex}{\SavedStyle-}{%
    \stackinset{c}{.2\LMpt}{c}{.5\LMpt-.2\LMex}{\SavedStyle-}{%
      \SavedStyle\phantom{\int}}}}\setbox0=\hbox{$\SavedStyle\int\,$}\kern-\wd0}\int}

\theoremstyle{plain}

\def \eps{\varepsilon}
\def \R{\mathbb R}

\begin{document}

%- Theorems and similar stuff: ------

\theoremstyle{plain}
\newtheorem{theorem}{Theorem} [section]
\newtheorem{corollary}[theorem]{Corollary}
\newtheorem{lemma}[theorem]{Lemma}
\newtheorem{proposition}[theorem]{Proposition}
\newtheorem{example}[theorem]{Example}
\theoremstyle{definition}
\newtheorem{definition}[theorem]{Definition}%
\theoremstyle{remark}
\newtheorem{remark}{Remark}

%- Numeracion ------------------------------------
%\numberwithin{theorem}{section}
%\numberwithin{equation}{section}
\renewcommand{\theequation}{\arabic{section}.\arabic{equation}}

%%%%%%%%%%%%%%%%%%%%%%%%%%%%%%%%%%%%%%%%%%%%%%%%%%%%%%%%%%%%%%%%%%%%%%%%%%%%%%%%%%%%%%%%%%%%%%%%%%
\title{On the elasto-plastic filtration equation}

\author{Arturo de Pablo, Fernando Quir\'os and Julio D. Rossi}

\address{A.\,de Pablo.
        Departamento de Matem\'{a}ticas, Universidad Carlos III de Madrid, 28911-Legan\'{e}s, Spain \& Instituto de Ciencias Matem\'aticas ICMAT (CSIC-UAM-UC3M-UCM), 28049-Madrid, Spain.}
\email{arturop@math.uc3m.es}

\address{F.\,Quir\'os.
        Departamento de Matem\'{a}ticas, Universidad Aut\'onoma de Madrid,
        \& Instituto de Ciencias Matem\'{a}ticas ICMAT (CSIC-UAM-UCM-UC3M),
        28049-Madrid, Spain.}
\email{fernando.quiros@uam.es}

\address{J.\,D.\,Rossi.
        Departamento de Matem{\'a}tica y Estadistica,
        Universidad Torcuato di Tella,
        Av. Figueroa Alcorta 7350, (1428),
        Buenos Aires, Argentina.}
\email{julio.rossi@utdt.edu}

\newcommand{\pablo}[1]{{\color{blue}{#1}}}
\long\def\comment#1{\marginpar{\raggedright\small$\bullet$\ #1}}

\keywords{Elasto-plastic filtration equation, fully nonlinear equations, viscosity solutions, large-time behaviour, boundary layer.}

\subjclass[2020]{35K55, 35B30, 35B40, 35Q91, 35D40.}
    % 35K55 = Nonlinear parabolic equations
    % 35B30 = Dependence of solutions to PDEs on initial and/or boundary data and/or on parameters of PDEs
    % 35B40 = Asymptotic behavior of solutions to PDEs
    % 35Q91 = PDEs in connection with game theory, economics, social and behavioral sciences
    % 35D40 = Viscosity solutions to PDEs

\date{}

\begin{abstract}
    We study the fully nonlinear heat equation $b(\partial_tu)\partial_tu=\Delta u$ posed in a bounded domain with Dirichlet boundary conditions. Here $b(s)=b^-$ if $s<0$, $b(s)=b^+$ if $s>0$, $b^-\neq b^+$ being two positive constants. This equation models the flow of an elastic fluid in an elasto-plastic porous medium.  We are interested in the existence and uniqueness of viscosity solutions and in their asymptotic behaviour as $t\to\infty$ and when $b^-\to 0^+$ or $b^+\to +\infty$. We also characterize solutions of the problem as limits of a minimization dynamic game.
\end{abstract}

\maketitle

%%%%%%%%%%%%%%%%%%%%%%%%%%%%%%%%%%%%%%%%%%%%%
\section{Introduction}\label{sec.intro}

We study the Cauchy-Dirichlet problem associated to the elasto-plastic equation of filtration, that is,
\begin{equation} \label{eq.main.intro}
    \begin{cases}
    \displaystyle b(\partial_tu) \partial_tu (x,t) = \Delta u(x,t)& \text{in } \Omega \times (0,+\infty), \\
    \displaystyle u(x,t) = 0& \text{on } \partial \Omega \times (0,+\infty), \\
    \displaystyle u(x,0) = u_0(x)& \text{in } \Omega,
    \end{cases}
\end{equation}
where
\begin{equation} \label{eq.b.intro}
    b(s) = \begin{cases}
    b^-, & s< 0, \\
    b^+, & s\geq 0,
    \end{cases}
\end{equation}
for two positive constants  $b^-\neq b^+$ and a bounded smooth domain $\Omega\subset\R^N$, $N\ge1$. We fix for instance $b^+>b^->0$, the other case $0<b^+<b^-$ being analogous.

Problem~\eqref{eq.main.intro}--\eqref{eq.b.intro} was proposed by Barenblatt and Krylov in~\cite{Barenblatt-Krylov-1955} to model the flow of an elastic fluid in an elasto-plastic porous medium, under the assumption that the porous medium is irreversibly deformable. In such a flow, the pressure $p$, the density $\rho$, the porosity $m$ (the relative volume occupied in the medium by the pores), and the velocity $\mathbf{v}$ are governed by the physical laws:
\begin{align*}
    \text{Conservation of mass:} \quad &\partial_t (m\rho) + \nabla \cdot (\rho \mathbf{v}) = 0; \\
    \text{Darcy's law:}  \quad& \mathbf{v} = -\frac{\kappa}{\mu} \nabla p;\\
    \text{Equation of state of the fluid:}\quad & \frac{\rho}{\rho_0} = 1 + \beta_f (p - p_0);\\
    \text{Equation of state of the medium:} \quad &\partial_t m = m_0\beta_\mathfrak{m}\partial_t p.
\end{align*}
Here:
\begin{itemize}
    \item $\kappa$ is the coefficient of permeability, determining the resistance of the porous medium to the fluid leaking through it;
    \item $\mu$ is the coefficient of viscosity of the fluid;
    \item the constant $\beta_f>0$ is the coefficient of compressibility of the fluid (which is assumed to be weakly compressible);
    \item $p_0$ and $\rho_0$ are respectively the reference pressure and density of the fluid.
    \item $\beta_\mathfrak{m}>0$ is the coefficient of compressibility of the porous medium, given by
\[
    \beta_\mathfrak{m}=
    \begin{cases}
    \beta^- & \text{if } \partial_t p < 0,\\
    \beta^+ & \text{if } \partial_t p > 0,
    \end{cases}
\]
with $\beta^-$, $\beta^+$  positive constants, since the  medium is assumed to be elasto-plastic;
    \item $m_0$ is the reference value for the porosity.
\end{itemize}
Eliminating $\rho$, $m$, and $\mathbf{v}$ we obtain, to first order, that the excess fluid pressure $u=p-p_0$ is a solution to
\[
    b(\partial_t u)\partial_t u =\Delta u,\qquad\textup{where }
    b(\partial_t p)=
    \begin{cases}
        \displaystyle\frac{\mu m_0(\beta_f + \beta^-)}{\kappa}&\textup{if }\partial_t u < 0,\\[10pt]
        \displaystyle\frac{\mu m_0(\beta_f + \beta^+)}{\kappa} &\textup{if }\partial_t u > 0.
    \end{cases}
\]
Thus, $u$ is a solution to~\eqref{eq.main.intro}--\eqref{eq.b.intro}  with
\[
    b^- = \frac{\mu m_0(\beta_f + \beta^-)}{\kappa},\qquad b^+=\frac{\mu m_0(\beta_f + \beta^+)}{\kappa}.
\]
See~\cite{Barenblatt-Entov-Ryzhik-1990} for further details of the derivation of the equation.

\begin{remark}
    If we set
    \[
        \gamma = \frac{b^+ - b^-}{b^+ + b^-}, \qquad \tau = \frac{2t}{b^-+b^+},
    \]
    and then rename $\tau$ as $t$, we obtain that $u$ satisfies
    \begin{equation}\label{eq:version.gamma}
        \partial_t u+\gamma|\partial_t u|= \Delta u, \qquad \gamma\in(-1,1),
    \end{equation}
    a version of the elasto-plastic filtration equation often found in the literature.
\end{remark}

Problem~\eqref{eq.main.intro}--\eqref{eq.b.intro}  is as well a particular case of a parabolic Bellman equation of dynamic programming
\begin{equation}\label{Bellman}
    \partial_tu=\min\{L_1u,L_2u\},
\end{equation}
where $L_1u=\Delta u/b^-$, $L_2u=\Delta u/b^+$. The existence of classical solutions for equations of this type has been proved in~\cite{Evans-Lenhart}. Alternatively, we show here the existence of a viscosity solution using Perron's method, see Theorem~\ref{teo-1.intro}.
In Section~\ref{sect-games} we will arrive to Problem~\eqref{eq.main.intro}--\eqref{eq.b.intro} yet from a third point of view, as the viscosity limit of a probabilistic game. A general reference for this kind of approach is~\cite{Blanc-Rossi-2019}.

We will deal here with homogeneous Dirichlet boundary conditions, $u=0$ on $\partial \Omega$, but more general data can be also considered, see Section~\ref{sect-basic}.

We will always work with viscosity solutions, which are defined in the following section. We settle the basic theory for Problem~\eqref{eq.main.intro}--\eqref{eq.b.intro}, including a comparison principle, and show that the classical solutions obtained in~\cite{Evans-Lenhart} are viscosity solutions. Finally,  we establish continuous dependence and a monotonicity property with respect to the parameters.

We denote $u_{b^-,b^+}$ the solution to Problem~\eqref{eq.main.intro}--\eqref{eq.b.intro} when we want to emphasize the dependence on the parameters. Some properties of the solution in terms of the parameters, such as continuity and monotonicity, are shown in Section~\ref{sect-basic}.

There are two interesting limit cases, namely $b^-\to0^+$ and $b^+\to+\infty$. The first one corresponds to the case in which the compressibility of the fluid is much smaller than that of the porous medium, which exhibits very little elastic recovery. In the second one, the compressibility of the porous medium under increasing fluid pressure is very big. These limit cases will be described in terms of the projection of the initial condition $u_0$ into an effective initial datum $\widetilde{u}_0$, given by the solution to the obstacle problem (from above) for the Laplacian,
\begin{equation} \label{eq.tilde.u0.intro}
    \widetilde{u}_0(x)=\sup\Big\{w (x):w\leq u_0,\; -\Delta w \leq 0 \text{ in } \Omega \textup{ in the viscosity sense}, \mbox{ and }  w\big|_{\partial\Omega}\leq0\Big\}.
\end{equation}

We first study the limit when $b^-$ goes to zero. We obtain that the limit is the solution to the elastoplastic equation with initial condition given by the projected datum $\widetilde{u}_0$.

\begin{theorem}\label{th:b1-0}
    If we fix $b^+>0$, then
    \begin{equation}
        \lim_{b^-\to0^+}u_{b^-,b^+}(x,t)=\widetilde{u}(x,t),
    \end{equation}
    locally uniformly in $(0,+\infty) \times \overline{\Omega}$,   where $\widetilde{u}$ is the unique solution to the problem
    \begin{equation} \label{eq.main.intro.proj}
        \begin{cases}
        \displaystyle  b^+ \partial_t \widetilde{u} = \Delta \widetilde{u} & \textup{in }\Omega \times (0,+\infty),\\
        \displaystyle \widetilde{u}  = 0&  \textup{on }\partial \Omega \times (0,+\infty),\\
        \displaystyle  \widetilde{u} (\cdot,0^+)  = \widetilde{u}_0&  \textup{in } \Omega.
        \end{cases}
    \end{equation}
    Therefore, $\partial_t\widetilde{u}\geq0$ in $\Omega$, whence $\widetilde{u}$ also solves the elasto-plastic equation.
\end{theorem}

This result tells us that for $b^-$ small, but positive, there is a very fast transition between the initial value $u_0$ and the initial trace for the limit $\widetilde{u}_0$. Our next result describes this transition layer connecting $\widetilde u(\cdot,0)=u_0$ and $\widetilde u(\cdot,0^+)=\widetilde{u}_0$.

\begin{theorem} \label{teo.layer}
    Let $${w}(x,t):=\lim\limits_{b^-\to0^+} u_{b^-,b^+} (x,b^- t).$$ Then, $w$ solves
    \begin{equation} \label{eq.main.intro.proj-1}
        \begin{cases}
            \displaystyle  \partial_t w =\min\{ \Delta w,0\} & \textup{in }\Omega \times (0,+\infty),\\
            \displaystyle w = 0&  \textup{on }\partial \Omega \times (0,+\infty),\\
            \displaystyle  w (\cdot,0)  = u_0&  \textup{in } \Omega,
        \end{cases}
    \end{equation}
    and it holds that $\lim\limits_{t \to \infty}{w}(\cdot,t)=\widetilde{u}_0$ uniformly in $\overline\Omega$.
\end{theorem}

\begin{remark}
    Theorem~\ref{teo.layer} gives the solution of the elliptic obstacle problem for the Laplacian as the asymptotic profile of an evolution governed by a fully nonlocal operator. This idea can by used to approximate numerically solutions of the former problem.
\end{remark}

Let us now give the limit when $b^+$ goes to infinity, whose proof follows from the arguments used to prove Theorems~\ref{th:b1-0} and~\ref{eq.main.intro.proj-1}.

\begin{theorem}\label{th:b2-inf}
    If  we fix $b^->0$, then
    \begin{equation*}
        \lim_{b^+\to\infty}u_{b^-,b^+}(x,t)=\bar{u} (x,t)
    \end{equation*}
    locally uniformly in $(0,+\infty) \times \overline{\Omega}$, with $\overline{u}$ the solution to
     \begin{equation} \label{eq.main.intro.proj-1.00}
        \begin{cases}
            \displaystyle  b^-\partial_t \overline{u} =\min\{ \Delta \overline{u},0\} & \textup{in }\Omega \times (0,+\infty),\\
            \displaystyle \overline{u}  = 0&  \textup{on }\partial \Omega \times (0,+\infty),\\
            \displaystyle  \overline{u} (\cdot,0)  = u_0&  \textup{in } \Omega.
        \end{cases}
    \end{equation}
    Moreover,  $\lim\limits_{t\to\infty}{\overline{u}}(\cdot,t)=\widetilde{u}_0$ uniformly in $\overline\Omega$.
\end{theorem}

As we will see later, solutions decay to zero as $t$ goes to infinity, $\lim\limits_{t\to\infty}u(x,t)=0$. To match this large-time behaviour with that of $\overline u$ we include the following result.

\begin{theorem} \label{teo.lll}
    If  we fix $b^->0$, then
    $$
        \lim\limits_{b^+\to\infty}u_{b^-,b^+}(x,b^+t)=\widehat{u} (x,t)
    $$
    locally uniformly in $(0,+\infty) \times \overline{\Omega}$, with $\widehat{u}$ the solution to
    \[
        \begin{cases}
            \displaystyle\partial_t\widehat{u}=\Delta \widehat{u} & \textup{in }\Omega \times (0,+\infty),\\
            \displaystyle\widehat{u}=0&\textup{on }\partial \Omega \times (0,+\infty),\\
            \displaystyle\widehat{u}(\cdot,0^+)=\widetilde{u}_0&  \textup{in } \Omega.
        \end{cases}
    \]
    Moreover,  $\lim\limits_{t\to\infty}{\widehat{u}}(\cdot,t)=0$ uniformly in $\overline\Omega$.
\end{theorem}

Notice that $\Delta \widetilde{u}_0 \geq 0$ and hence $\partial_t\widehat{u}\geq0$ in $\Omega$, and then we also have that $\widehat{u}$ solves the elasto-plastic equation, this time with $b^+=1$.

We prove these four theorems concerning the limit cases in Section \ref{sect-limit}.

Our next task is to study the asymptotic behaviour as $t\to\infty$ of $u_{b^-,b^+} (x,t)$. This behaviour will depend on the parameters and the initial datum $u_0$. To this end let $\{\varphi,\lambda\}$ be the first eigenfunction (normalized, positive with maximum 1) and the first eigenvalue of the Dirichlet Laplacian in $\Omega$, that is, we have
$$
    -\Delta\varphi=\lambda\varphi\quad\textup{in }\Omega,\qquad\varphi=0\quad\textup{on }\partial\Omega.
$$
Denote $\lambda_1 =\lambda/b^-$ and $\lambda_2 =\lambda/b^+$. Then, for any two positive constants $c_1,c_2$, the functions
$$
    u_1 (x,t) = c_1 e^{-\lambda_1 t} \varphi (x),\qquad u_2 (x,t) = - c_2 e^{-\lambda_2 t} \varphi (x),
$$
are both solutions to the equation in Problem~\eqref{eq.main.intro}. By a comparison argument, there exist positive constants $C_1,C_2$ such that
\begin{equation}\label{comportamiento}
    -C_2 e^{-\lambda_2 t} \varphi (x) \leq u(x,t) \leq C_1 e^{-\lambda_1 t} \varphi(x),
\end{equation}
for $x\in \overline{\Omega}$, $t>0$. The comparison can be done assuming that $|\nabla u_0|<\infty$ on the boundary. This is no restriction when considering the large-time behaviour, since we can take as new initial datum $u(\cdot,\varepsilon)$ for $\varepsilon>0$, which satisfies the desired assumption by the regularization effect of the heat equation.

The bounds in \eqref{comportamiento} cannot be improved for general data. Let us get better estimates when imposing different conditions on $u_0$.

If $u_0>0$ in $\Omega$, using comparison from above and from below with $k\varphi$ and $k'\varphi$ for some constants $k,k'$, we deduce
\begin{equation}\label{eq:starting.point.best.fitting}
    k\varphi(x)\le e^{\lambda_1 t} u(x,t) \leq k' \varphi(x).
\end{equation}
A \emph{best fitting} argument allows to prove the asymptotic behaviour in this case. In the same way we can deal with nonpositive initial data.

\begin{theorem}\label{u0>0}
  If $u_0>0$ in $\Omega$, then there exists some constant $K_1>0$ such that
  \begin{equation}\label{eq:lambda1}
    \lim_{t\to\infty} e^{\lambda_1 t} u(x,t)=K_1\varphi(x).
  \end{equation}
  If $u_0<0$ in $\Omega$, then there exists some constant $K_2>0$ such that
  \begin{equation}\label{eq:lambda2}
    \lim_{t\to\infty} e^{\lambda_2 t} u(x,t)=-K_2\varphi(x).
  \end{equation}
\end{theorem}

Notice that the previous result says that when $u_0$ does not change sign then solutions to our equation behave like solutions to the standard heat equation $b_i u_t = \Delta u$ with $b_i = b^-$ or $b^+$ depending on the sign of $u_0$.

On the other hand, it is clear that $u_{b^-,b^+}\le u_{b^+,b^+}$, see Proposition~\ref{prop-<heat}. This last function, as a solution to the standard heat equation, becomes negative in finite time provided the projection of the initial value over the first eigenfunction is negative.  Thus, we get the following result.

\begin{theorem}\label{u0 casi<0}
    If $\displaystyle\int_\Omega u_0 (x) \varphi(x) \,{\rm d}x\le0$ then \eqref{eq:lambda2} holds.
\end{theorem}

We consider next the case of \emph{bad initial values}, i.e. functions satisfying
\begin{equation}\label{bad-u0}
    u_0(x)<0 \quad \text{in some subdomain, and}\quad \int_\Omega u_0 (x) \varphi (x) \,{\rm d}x>0.
\end{equation}
In that case we may have both asymptotic behaviours \eqref{eq:lambda1} or \eqref{eq:lambda2}, depending on the parameters $b^-$ and $b^+$, and also other possible behaviours, with a faster decay, are possible. Let $\theta=b^+/b^-\in(1,\infty)$.

\begin{theorem}\label{th:bad-u0}
    Given $u_0$ satisfying \eqref{bad-u0}, there exists a critical value  $\theta^*\in(1,\infty)$ such that if $1<\theta<\theta^*$ then \eqref{eq:lambda1} holds for some $K_1(\theta)>0$, while if $\theta>\theta^*$ then \eqref{eq:lambda2} holds for some $K_2(\theta)>0$. If $\theta=\theta^*$ then the solution changes sign for every $t>0$.
    %In addition, there are initial conditions in 1-d for which solutions decay faster than eny given rate.
\end{theorem}

We prove these three theorems concerning the asymptotic behaviour in Section~\ref{sect-asymp}. We complement these results with explicit one-dimensional examples of bad initial data corresponding to $\theta=\theta_*$, and show that in this case larger decays are possible for some special initial conditions. These different behaviours reflect the strong nonlinearity that involves $\partial_t u$ in our equation.

Let us compare our large-time behaviour results with the one for the same equation when the spatial domain is the whole space, a problem that has received some attention over the years; see, for instance;~\cite{Barenblatt-Sivashinskii-1969,Kamin-1957,Kamin-1969,KPV91}.  For this Cauchy problem, with the equation written in the version~\eqref{eq:version.gamma}, assuming an exponential decay at infinity of the initial value, it was proved in~\cite{KPV91} that for each $\gamma~\in(-1,1)$ there exist a unique exponent $\alpha=\alpha(\gamma)$ and an integrable radially symmetric and smooth positive function~$F$ such that
$$
    \lim_{t\to\infty}\sup_{|x|\le c\sqrt t}t^\alpha|u(x,t)-B(x,t)|=0\quad\textup{for all }c>0,\quad\textup{where }B(x,t)=t^{-\alpha}F(x/\sqrt t).
$$
The similarity exponent $\alpha$, that gives the decay rate of all solutions, cannot be determined from dimensional considerations, and has to be obtained, together with $F$, as part of the solution of a nonlinear eigenvalue problem. Such similarity exponents $\alpha=\alpha(\gamma)$ are known in the literature as anomalous exponents, and the corresponding solution $B$ as a a self-similar solution of second kind; see, for instance~\cite{Barenblatt-book-1996}. In our case the function $\alpha$ is known to be strictly increasing and to satisfy $\alpha(0)=N$. Hence $B$ is of source type (finite nontrivial initial mass) only when $\gamma=0$,  very singular (infinite initial mass) if $\gamma>0$ and mildly singular (zero initial mass) if $\gamma<0$. These and other properties of $\alpha$ can be found in~\cite{Aronson-Vazquez-1994,Barenblatt-Sivashinskii-1969,Goldenfeld-Martin-Oono-Liu-1990,KPV91,Peletier-1995}.
Large-time geometrical properties have been also obtained for this problem; see \cite{Huang-Vazquez-2012}.

A more recent paper dealing with the large-time behaviour of a large class of fully nonlinear parabolic equations, with attracting self-similar solutions of second kind is~\cite{Armstrong-Trokhimtchouk-2010}. Self-similarity also plays a role in the construction of a counterexample for higher regularity to solutions of the equation in~\cite{Caffarelli-Stefanelli-2008}.

Notice that, in contrast with the case of the whole space, when the problem is set in a bounded domain, both the decay rate and the asymptotic profile are not universal, and depend strongly on the initial condition.

We end the paper using a game theoretical approach in order to obtain Problem~\eqref{eq.main.intro}--\eqref{eq.b.intro} from a point of view that lies very far from the classical elasto-plastic physical application. We obtain it as a limit, in the viscosity sense, of a dynamic minimization game. Let us describe briefly the game; more details are given in Section~\ref{sect-games}. Starting at some point in a given domain $\Omega$, and given some time, a particle moves at each step randomly at a distance less than $\varepsilon$, and a player (who wants to minimize the expected outcome of the game) chooses a time lapse in the interval $[C b^-\varepsilon, C b^+\varepsilon]$ (here $C$ is an explicit constant that depends only on the dimension). The game ends when the particle exits the set $\Omega$, or when the time is over, and the final payoff is then given by some fixed function in each case. The value of the game is then defined as
\begin{equation} \label{value.eq.intro}
    u^\varepsilon (x,t) = \inf_{S} \mathbb{E}_S^{(x,t)}(\mbox{final payoff}),
\end{equation}
where the infimum is taken among all possible strategies $S$ (choices of the time lapse at every round of the game) for the player.

In Section~\ref{sect-games} we show that this infimum exists and satisfies the Dynamic Programming
Principle~(DPP)
\begin{equation} \label{eq.main.intro2}
    \begin{cases}
        \displaystyle u^\varepsilon (x,t)=\inf_{b\in [C b^-, C b^+]}\dashint_{B_{\varepsilon}(x)}\hspace{-3mm}u^\varepsilon (y, t - b \varepsilon^2)\,{\rm d}y,& (x,t)\in \Omega \times (0,+\infty), \\
        \displaystyle u^\varepsilon (x,t) = 0, &(x,t)\in D_\eps\times (0,+\infty), \\
        \displaystyle u^\varepsilon (x,t) = u_0(x), \quad & (x,t)\in\Omega \times [-Cb^+\varepsilon^2,0].
    \end{cases}
\end{equation}
where $D_\eps=(\Omega+B_\varepsilon)\setminus\Omega$. We also prove that taking the limit $\varepsilon\to0^+$ we get a solution, in the viscosity sense, of Problem~\eqref{eq.main.intro}--\eqref{eq.b.intro}. We observe that if we reverse the constants, i.e, we take $0<b^+<b^-$, in the corresponding game we maximize the final payoff.

\begin{theorem} \label{teo.juegos.intro}
    Take $C=c(N)=\dfrac{1}{2N(N+2)}$. Given a continuous function $u_0:\overline{\Omega} \mapsto \mathbb{R}$ with $u_0\big|_{\partial\Omega}=0$, let $u$ the unique viscosity solution to Problem~\eqref{eq.main.intro}--\eqref{eq.b.intro} and $u^\varepsilon$ the unique solution to Problem~\eqref{eq.main.intro2}. Then,
    \begin{equation}\label{conv.intro}
        \lim_{\varepsilon\to0}u^\varepsilon= u
    \end{equation}
    uniformly in compact subsets of $\overline{\Omega} \times [0,\infty)$.
\end{theorem}

Alternatively, looking at the problem with the nonlinear coefficient in front of the diffusion term, as in~\eqref{Bellman}, that is,
\begin{equation} \label{eq.main.intro.99}
    \begin{cases}
    \displaystyle \partial_tu=\varphi(\Delta u)\Delta u&\textup{in }\Omega\times(0,+\infty),\\
    \displaystyle u=0&\textup{on }\partial\Omega\times(0,+\infty),\\
    \displaystyle u=u_0&\textup{in }\Omega,
    \end{cases}
\end{equation}
with $\varphi(s)=1/b(s)$, we can formulate another approximating game whose DPP is given by,
\begin{equation} \label{eq.main.alternate}
    \begin{cases}
        \displaystyle u^\varepsilon(x,t)=\inf\!\left\{\!\dashint_{B_{\sigma_1}(x)}\hspace{-3mm}u^\varepsilon (y,t-C\varepsilon^2)\,{\rm d}y,\dashint_{B_{\sigma_2}(x)}\hspace{-3mm}u^\varepsilon (y, t - C \varepsilon^2)\,{\rm d}y\right\}\!\!,& (x,t)\in \Omega \times (0,+\infty), \\
        \displaystyle u^\varepsilon (x,t) = 0, &(x,t)\in D_\eps \times (0,+\infty), \\
        \displaystyle u^\varepsilon (x,t) = u_0(x), \quad & (x,t)\in\Omega \times [-Cb^+\varepsilon^2,0],
    \end{cases}
\end{equation}
where $\sigma_i=\varepsilon/\sqrt{b_i}$, for $i=1,2$. This problem appears as the DPP associated with the following game, at each turn time decreases by $C \varepsilon^2$ (here $C$ depends only on the dimension) and a player chooses to sort the next position of the game with uniform probability inside two possible balls of radii $\varepsilon/\sqrt{b^-}$ and $\varepsilon/\sqrt{b^+}$. The game ends when the position leaves the domain, with final payoff equal to zero, or when time expires, with final payoff given by $u_0$. For this game one can also show that its value functions, that are
solutions to \eqref{eq.main.alternate}, also converge to the unique solution to Problem~\eqref{eq.main.intro}--\eqref{eq.b.intro}. The details of this convergence are left to the reader.

\medskip

\noindent\emph{Organization of the paper.} In Section \ref{sect-basic} we gather some definitions and preliminary results, including the basic theory related to Problem~\eqref{eq.main.intro}--\eqref{eq.b.intro}. We describe the limit cases of the parameters $b_i$ in Section~\ref{sect-limit}. Section~\ref{sect-asymp} is devoted to study the large-time behaviour, the limit profile as $t\to\infty$ in terms of the initial value. Finally, in Section \ref{sect-games} we deal with the game approximating the elasto-plastic equation.

%%%%%%%%%%%%%%%%%%%%%%%%%%%%%%%%%%%%%%%%%%%%%%%%%%%%%%%%%%%%%%%%%%%
\section{Preliminaries. Basic theory} \label{sect-basic}
\setcounter{equation}{0}

As we commented in the Introduction, if we consider a boundary datum $g=g(x)$ we can subtract the harmonic extension $v$ of $g$, satisfying
\begin{equation} \label{harmonic.intro}
    \Delta v=0 \quad\text{in }\Omega, \qquad v=g\quad \text{in }\partial \Omega,
\end{equation}
which is the unique stationary solution to our evolution problem. Then, $w=u-v$ solves
\begin{equation} \label{eq.main.intro.g=0}
    \begin{cases}
        \displaystyle  b(\partial_tw) \partial_tw = \Delta w&\text{in } \Omega \times (0,+\infty), \\
        \displaystyle w  = 0&\text{in } \partial \Omega \times (0,+\infty), \\
        \displaystyle w(\cdot,0)  = u_0 - v&\text{in }\Omega.
    \end{cases}
\end{equation}
If now $g=g(x,t)$, with an analogous procedure what we get is the Bellman problem
\begin{equation} \label{eq.main.intro.gt=0}
    \begin{cases}
        \displaystyle  \partial_t w = \min\{L_1w,\,L_2w\}+G&\text{in } \Omega \times (0,+\infty), \\
        \displaystyle w= 0& \partial \Omega \times (0,+\infty), \\
        \displaystyle w(\cdot,0)  = u_0 - G(\cdot,0)&\text{in } \Omega,
    \end{cases}
\end{equation}
where $G=-\partial_tv $, $v$ being the harmonic extension, for each $t>0$, of $g(\cdot,t)$. We do not consider this last case, and by the previous discussion, we can assume, to simplify and without loss of generality, that $g\equiv0$.

Let us state the precise definition of being a viscosity solution to our problem.

Let $T>0$. A bounded upper (resp. lower) semicontinuous function $u\colon\overline{\Omega}\times[0,T]\to\mathbb{R}$ is said to be a viscosity subsolution  (resp. supersolution) of
\begin{equation}\label{ev_eq_T}
    \begin{cases}
        b(\partial_tu) \partial_tu = \Delta u &\text{in }\Omega\times(0,T),\\
        u=0 &\text{on }\partial \Omega \times(0,T),\\
        u(x,0)=u_0(x) &\text{in }\Omega,
	\end{cases}
\end{equation}
if
\begin{itemize}
	\item $u(x,t)\leq 0$ \mbox{ {(resp. $\geq$)}} on $\partial\Omega \times(0,T);$
	\item $u(x,0)\leq u_0(x)$ \mbox{ {(resp. $\geq$)}} in $\Omega;$
    \item For every smooth function $\varphi\colon\Omega\times[0,\infty)\to \mathbb{R}$,	at every maximum (resp. minimum) point $(x_0,t_0)\in\Omega\times(0,T)$ of $u-\varphi$ in $B_\delta(x_0)\times(t_0-\delta,t_0+\delta)\subseteq\Omega\times (0,T)$ for some $\delta>0$ we have that
\end{itemize}
\begin{equation}\label{def-viscosity-sol}
	b(\partial_t\varphi)\partial_t\varphi(x_0,t_0)-\Delta\varphi(x_0,t_0)\le0\qquad\text{(resp. } \ge0).
\end{equation}
	
\begin{definition}
    A bounded function $u\colon \Omega\times[0,T]\to\mathbb{R}$ is a viscosity solution of \eqref{ev_eq_T} if it is both viscosity subsolution and supersolution. A continuous function $u\colon \overline{\Omega}\times[0,\infty)\to\mathbb{R}$ is a viscosity solution if it is a viscosity solution for every $T>0$.
\end{definition}

We start with an easy comparison which will produce some useful barriers. Let $U_\kappa(f)$ be the solution of the standard heat Cauchy-Dirichlet problem with diffusivity $1/\kappa$, $\kappa>0$, and initial value $f$, i.e. $v=U_\kappa(f)$ satisfies
\begin{equation} \label{eq.U}
    \begin{cases}
        \displaystyle \kappa \partial_tv (x,t) = \Delta v(x,t)&\textup{in } \Omega \times (0,+\infty), \\
        \displaystyle v(x,t) = 0&\textup{on } \partial \Omega \times (0,+\infty), \\
        \displaystyle v(x,0)=f(x)&\textup{in } \Omega.
    \end{cases}
\end{equation}
Since $b^+ > b^-$, we have that $b(s) s= \max\{b^- s,b^+s\}$ for every $s \in \mathbb{R}$, so that $u_{b^-,b^+}$ is a subsolution of the heat equations with diffusivities $\kappa=b^-,b^+$. This yields immediately the following estimates.

\begin{proposition}\label{prop-<heat} It holds that
  $$
  u_{b^-,b^+}\le\min\big\{U_{b^-}(u_0),U_{b^+}(u_0)\big\}.
  $$
  The inequality is strict if $\partial_t u_{b^-,b^+}$ changes sign.
\end{proposition}

\begin{theorem}\label{teo-1.intro}
    Let $u_0$  be a continuous function in $\overline\Omega$ with $u_0 \big|_{\partial \Omega } =0$, then there exists a unique viscosity solution to \eqref{eq.main.intro}. This solution is continuous in $\overline{\Omega} \times [0,T]$. Moreover, a comparison principle holds, if $\overline{u}$ is a supersolution and $\underline{u}$ is a subsolution to the equation such that $\overline{u}\geq\underline{u}$ at the parabolic boundary $\Gamma_{T}=(\Omega\times\{0\})\cup(\partial\Omega\times(0,T))$, then $\overline{u}\geq\underline{u}$ in the whole $\overline{\Omega}\times [0,T]$.
\end{theorem}	

The proof is obtained as usual: first, we show a comparison principle and then we conclude by Perron's method using adequate sub and supersolutions.

\begin{theorem} \label{comparacion}
	Let $T>0$. Let $u$ and $v$ be a viscosity subsolution and a supersolution of~\eqref{ev_eq_T} respectively. Then, $u\le v$ on $\Omega\times [0,T]$.
\end{theorem}
	
\begin{proof}
		Following ideas given in \cite{CIL}, we consider the function
		$$
			z(x,t)\coloneqq u(x,t)-
			\frac{\delta}{T-t},\qquad\delta>0,
		$$
		that satisfies, in the viscosity sense,
	  		$$
			\begin{cases}
		 b(\partial_t z) \partial_t z - \Delta z\leq-
		\frac{\delta b^-}{(T-t)^2}<0 &\text{in }\Omega\times(0,T),\\			
		 z (x,t)= - \frac{\delta}{T-t} &\text{on }
			\partial \Omega \times(0,T),\\
		z (x,0)=u_0(x) - \frac{\delta}{T} &\text{in }\Omega, \\
		z(x,t)\to -\infty, &\text{when }
			t\to T^{-}.\\
			\end{cases}
	$$
	It is clear that, if we prove $z(x,t)\leq v(x,t)$ for any $x\in\Omega$ and $t\in(0,T)$ then the conclusion follows by taking $\delta\to 0$.
				
	We argue by contradiction. Let us assume that
	$$
		\sup_{x\in\overline{\Omega},\, t\in [0,T)}{z(x,t)-v(x,t)}\coloneqq\eta>0.
	$$
    Since $z$ and $-v$ are bounded from above, $\overline\Omega$ is compact and $z\to -\infty$ as $t\to T^{-}$ then, for some $\alpha$ big enough that will be chosen later, there exists $(x_{\alpha}, y_{\alpha}, t_{\alpha}, s_{\alpha})$ the maximum point of
	$$
		G_{\alpha}(x,y,t,s)\coloneqq z(x,t)-v(y,s)-\alpha(|x-y|^{2}+|t-s|^2),
	$$
    with $(t,s)\in [0,T)^2$, $(x,y)\in \overline{\Omega}\times \overline{\Omega}$. If we denote
	$$
		M_{\alpha}\coloneqq	G_{\alpha}(x_{\alpha}, y_{\alpha}, t_{\alpha}, s_{\alpha})(\geq \eta>0),
	$$
	it is clear that $M_{\alpha}\searrow \widetilde{M}$ as $\alpha\to\infty$, for some $\widetilde{M}\geq\eta>0$. Moreover, since
	$$
		M_{\frac{\alpha}{2}}\geq G_{\frac{\alpha}{2}}(x_{\alpha},y_{\alpha},t_{\alpha},s_{\alpha})=M_{\alpha}
        +\frac{\alpha}{2}(|x_{\alpha}-y_{\alpha}|^{2}+|t_{\alpha}-s_{\alpha}|^2),
	$$
	it follows that
	$$
		\lim_{\alpha\to\infty}\alpha |x_{\alpha}-y_{\alpha}|^{2}=\lim_{\alpha\to\infty}		\alpha |t_{\alpha}-s_{\alpha}|^{2}= 0.
	$$
	We notice that the previous limits also imply, taking a subsequence, that if
	\[
		\lim_{\alpha\to\infty}(x_\alpha,y_\alpha)\coloneqq(\widetilde x,\widetilde y)\in \overline{\Omega}\times \overline{\Omega},\text{ and }
		\lim_{\alpha\to\infty}(t_\alpha,s_\alpha)=(\widetilde t, \widetilde s)\in [0,T)\times [0,T),
	\]
	then $\widetilde x=\widetilde y$ and $\widetilde t=\widetilde s$. We also observe that, since
	$$
        0<\eta\leq\widetilde{M}=\lim_{\alpha\to\infty}G_{\alpha}(x_{\alpha},y_{\alpha}, t_{\alpha}, s_{\alpha})\leq z(\widetilde x, \widetilde t)-v(\widetilde x,\widetilde t)\leq \eta,
	$$
	clearly
	\begin{equation}\label{mg}
	   \eta=z(\widetilde x, \widetilde t)-v(\widetilde x, \widetilde t).
	\end{equation}
	By using the facts that $z(\cdot,t)<0\leq v(\cdot,t)$ on $\partial\Omega$ and $z(x,0)\leq v(x,0)$,~\eqref{mg} implies that
	$\widetilde x=\widetilde y\notin \partial\Omega$ and $\widetilde t=\widetilde s\neq 0$.
	
		We define now the regular function
	$$
		\psi_{\alpha}(\zeta,t)\coloneqq v(y_{\alpha}, s_{\alpha})+\alpha(|t-s_{\alpha}|^2+|\zeta-y_{\alpha}|^2),
	$$
	that clearly satisfies
	\begin{equation}\label{entiempo}
		\partial_{t}\psi_{\alpha}(x_{\alpha}, t_{\alpha})=2\alpha(t_\alpha-s_\alpha).
	\end{equation}
    Moreover $\psi_{\alpha}(\zeta,t)$ touches $z$ from above at $(x_\alpha,t_{\alpha})$. Therefore, since $z$ is a subsolution, by~\eqref{entiempo} it follows that
	\begin{equation}\label{uno}
	   0>-\frac{\delta}{(T-t)^2}\geq b(2\alpha(t_\alpha-s_\alpha)) 2\alpha(t_\alpha-s_\alpha)- 2N.
	\end{equation}
	
	Defining now
	$$
		\varphi_{\alpha}(\zeta,t)\coloneqq u(x_{\alpha}, t_{\alpha})-\alpha(|t-t_{\alpha}|^2+|\zeta-x_{\alpha}|^2),
	$$
    reasoning in a similar way, we also get that $\varphi_{\alpha}(\zeta,t)$ touches $v$ from below at $(y_\alpha,s_{\alpha})$. Therefore, since it is also true that $\partial_t\varphi_{\alpha}(y_{\alpha},s_{\alpha})=-2\alpha(s_\alpha-t_\alpha)$, using now the fact that $v$ is supersolution, we obtain
	\begin{equation}\label{dos}
		0 \leq b(2\alpha(t_\alpha-s_\alpha)) 2\alpha(t_\alpha-s_\alpha) -2N.
	\end{equation}
	Therefore, by \eqref{uno} and \eqref{dos}, we get
	$$
	   0 \leq b(2\alpha(t_\alpha-s_\alpha)) 2\alpha(t_\alpha-s_\alpha) -2N \leq  - \frac{\delta}{(T-t_\alpha)^2} <0,
	$$
	which gives the desired contradiction.
\end{proof}
	
With the previous result the classical Perron's method can be used to obtain the existence of a solution.

\begin{proof}[Proof of Theorem \ref{teo-1.intro}]
    We use as upper barrier	$\overline{w}=U_{b^-}(u_0)$ and as lower barrier $\underline{w}=U_{b^+}(\widetilde u)$, the solution to~\eqref{eq.main.intro.proj} with datum~\eqref{eq.tilde.u0.intro}. Notice that both $\overline{w}$ and $\underline{w}$ attain the data	in a classical continuous way. Assuming that the data are compatible, we obtain a (viscosity) subsolution and a supersolution of our Problem~\eqref{ev_eq_T} that are continuous in~$\overline{\Omega}$ and well ordered . Then, Perron's method gives us the existence of a viscosity solution to \eqref{ev_eq_T}. Uniqueness follows from the comparison principle.
\end{proof}

As a consequence of the comparison principle, we obtain a monotonicity property with respect to the parameters.

\begin{proposition}
    If $b^-<\hat b^-$, then $u_{b^-,b^+}\le u_{\hat b^-,b^+}$, while if $b^+<\hat b^+$, then $u_{b^-,b^+}\ge u_{b^-,\hat b^+}$.
\end{proposition}

\begin{proof}
    Let
    $$
        b(s) =
        \begin{cases}
            b^-,&s< 0,\\
            b^+,&s\geq0,
        \end{cases}
        \qquad\hat b(s) =
        \begin{cases}
            \hat b^-,&s<0,\\
            b^+,& s\geq0.
        \end{cases}
    $$
    If $b^-<\hat b^-$, then $b(s)\le \hat b(s)$ for all $s\in\mathbb{R}$. Therefore, $u_{b^-,b^+}$ is a subsolution to the problem with $\hat b$ and we conclude that $u_{b^-,b^+}\le u_{\hat b^-,b^+}$.

    The other case is similar.
\end{proof}

Also as a consequence of the comparison principle together with the continuity of the solutions, we obtain a continuous dependence with respect to the parameters.

\begin{proposition}
    The solution $u_{b^-,b^+}$ depends continuously on $b^-$ and $b^+$ in $(0,+\infty)$.
\end{proposition}

\begin{proof}
    Take a sequence $b^-_n \to b^-$. Now, take
    $$
        \hat{u} (x,t) = u_{b^-,b^+} (x, a_n t)
    $$
    with $a_n = b^-_n/b^-$. This function $\hat{u}$ verifies $\partial_t\hat{u}_t(x,t)=a_n \partial_t u_{b^-,b^+}(x, a_n t)$ and hence it solves $$ b_{b^-_n,b^+} (\partial_t \hat{u}_t ) \partial_t \hat{u}_t (x,t)=\Delta \hat{u}_t (x,t) $$ with $\hat{u}_t (x,0) = u_0 (x)$ and $\hat{u}_t  (x,t)=0$ for $(x,t)\in\partial\Omega\times[0,\infty)$. Hence, by uniqueness we get $\hat{u}(x,t)=u_{b^-,b^+}(x,a_n t)=u_{b^-_n,b^+}(x, t)$ and the proof concludes using the continuity of $u_{b^-,b^+}$ and that $a_n \to 1$.

    Continuous dependence with respect to $b^+$ can be obtained similarly.
\end{proof}
	
To end this section we observe that it can be easily checked that a classical solution (for this we mean a $C^{2,1}$ function that verifies the equation and the boundary conditions pointwise) is also a viscosity solution to our problem. In fact, for every smooth test function $\varphi\colon\Omega\times[0,\infty)\to \mathbb{R}$ at every maximum (resp. minimum) point $(x_0,t_0)\in\Omega\times (0,T)$ of $u-\varphi$ in $B_\delta(x_0)\times (t_0-\delta,t_0+\delta)\subseteq \Omega\times (0,T)$ for some $\delta>0$ we have that
$$	
    \partial_t\varphi(x_0,t_0)=\partial_t u(x_0,t_0)\quad\mbox{ and }\quad\Delta\varphi(x_0,t_0)\geq\Delta u(x_0,t_0)\,\quad\text{ (resp. }\le0),
$$
and we obtain that $u$ is a viscosity subsolution and a viscosity supersolution.

We refer to \cite{Evans-Lenhart} for an existence result for classical solutions to our problem. Here we prefer to work with viscosity solutions since passing to the limit in the viscosity setting is simpler.

%%%%%%%%%%%%%%%%%%%%%%%%%%%%%%%%%%%%%%%%%%%%%%%%%%%%%%%%%%%%%%%
\section{Limit cases} \label{sect-limit}
\setcounter{equation}{0}

We prove here Theorems \ref{th:b1-0} and \ref{th:b2-inf} concerning the cases $b^-\to0^+$ and $b^+\to\infty$ as well as the description of the transition layers, Theorems \ref{teo.layer} and \ref{teo.lll}.

\begin{proof}[Proof of Theorem \ref{th:b1-0}]
    Let $\widetilde u$ be the unique solution to Problem~\eqref{eq.main.intro.proj}, and take $w$ the solution to
    \[
        \begin{cases}
            b^+\partial_t w=\Delta w&\textup{in }\Omega\times (0,+\infty),\\
            w=0&\textup{on }\partial\Omega\times (0,+\infty),\\
            w(\cdot,0)=u_0&\textup{in }\Omega.
        \end{cases}
    \]
    Notice that $\widetilde u$ is a subsolution  to~\eqref{eq.main.intro} and $w$ a supersolution. Therefore, by comparison,
    \begin{equation}\label{eq:comparison.for.identification}
        \widetilde u(x,t)\le u_{b^-,b^+}(x,t)\le w(x,t).
    \end{equation}
    Since $u_{b^-,b^+}$ decreases when $b^-$ decreases, there is a limit $$z(x,t) =\lim\limits_{b^-\to0^+}u_{b^-,b^+}(x,t).$$ It is easy to check that $z$ solves
    \[
        \begin{cases}
            \widetilde b(\partial_t z)\partial_t z=\Delta z&\textup{in }\Omega\times (0,+\infty),\\
            z=0&\textup{on }\partial\Omega\times (0,+\infty),
        \end{cases}
    \]
    where $\widetilde b(s)=0$ if $s<0$ and $\widetilde b(s)=b^+$, in the viscosity sense. On the other hand, passing to the limit in~\eqref{eq:comparison.for.identification} we get
    \[
        \widetilde u(x,t)\le z(x,t)\le w(x,t).
    \]
    Notice that $\Delta z(x,t)\ge0$ and $\partial_t z(x,t)\ge 0$. In particular, the limit $z(x,0^+)=\lim\limits_{t\to0^+} z (x,t)$ exists, and $\Delta z(x,0)\ge0$. Finally, passing to the limit in the above estimates, we obtain
    \[
        \widetilde u_0(x)\le z(x,0)\le u_0(x),\quad x\in\Omega,
    \]
    and we conclude that $ \widetilde u_0(x) = z(x,0)$, since $\Delta z(x,0)\ge0$ and $ \widetilde u_0$ is the solution to the obstacle problem for the Laplacian with obstacle $u_0$.

    Summarizing, we have obtained that the limit $z$ coincides with $\widetilde{u}$. Uniform convergence follows from Dini's theorem. In fact, we have that $z(x,t)$ is a continuous function that is the pointwise limit of a decreasing sequence of continuous functions  $u_{b^-,b^+}(x,t)$ inside compact subsets of $\overline{\Omega} \times (0,+\infty)$.
\end{proof}

Next, let us describe the transition layer that describes the connection of $u_0$ at $t=0$ with $\widetilde u_0$ at $t=0^+$ as $b^- \to 0$.

\begin{proof}[Proof of Theorem \ref{teo.layer}]
    Let $V(x,t)=u_{b^-,b^+}(x,b^-t)$. Then,  $\Delta V(x,t)=\Delta u_{b^-,b^+}(x, b^- t)$ and $\partial_t V(x,t)=b^-\partial_t u_{b^-,b^+}(x, b^- t)$. Hence,
    \[
        \begin{cases}
            \partial_t V = A(\Delta V)&\textup{in }\Omega\times (0,+\infty),\\
            V =0&\textup{on }\partial\Omega\times (0,+\infty),\\
            V(\cdot,0)= u_0&\textup{in }\Omega,
        \end{cases}
        \qquad\textup{with }A(D) =
        \begin{cases}
            D, & D\leq 0, \\[8pt]
            \displaystyle \frac{b^-}{b^+} D, & D>0.
        \end{cases}
    \]
    The viscosity limit as $b^- \to 0^+$ of $V$ is a solution to
    $$
        \partial_t V = \min\{ \Delta V, 0 \}
    $$
    with boundary condition $V=0$ and initial datum $u_0$. This limit function $V$ is nonincreasing in time, and is bounded from below by $\widetilde{u}_0$, by comparison, since this latter function is a stationary solution of the problem which is initially below $u_0$.
    Hence, the pointwise limit $\lim_{t\to\infty}V(x,t)=z(x)$ exists. If we prove that $0\le\min\{\Delta z,0\}$ in the viscosity sense in $\Omega$, we are done. Indeed, this implies that $\Delta z\ge0$ in the viscosity sense, and since we also know that $z\le u_0$, because of the monotonicity of $u$, then due to the definition of $\widetilde u_0$, we have $z\le \widetilde u_0$, so that $z\equiv \widetilde u_0$.

   Because of the monotone convergence of $V(x,t)$ to $z(x)$, this latter function is upper semicontinuous. Let $\varphi$ be a smooth test function such that $\varphi-z$ attaches a strict minimum at $x_0\in\Omega$. Using again the convergence, there are points $\{x_{t}\}$ such that $x_t\to x_0$ as $t\to\infty$ and $\varphi-V(\cdot,t)$ attains a minimum at $x_t$. Hence, as $V$ is a viscosity solution to $V_t=\{\Delta V,0\}$, then $0\le\min\{\Delta\varphi(x_t),0\}$. Passing to the limit as $t\to\infty$, we conclude that  $0\le\min\{\Delta\varphi(x_0),0\}$, hence the result follows.

   Uniform convergence in $\overline{\Omega}$ follows now using again Dini's theorem.
\end{proof}

Now Theorems \ref{th:b2-inf} and \ref{teo.lll} follow as a consequence of our previous arguments.

\begin{proof}[Proof of Theorem \ref{th:b2-inf}]
We can write the problem satisfied by $u$ as
\begin{equation} \label{eq.main.intro.99.88}
    \begin{cases}
    \displaystyle \partial_tu= \min \big\{ \frac{1}{b^-}\Delta u, \frac{1}{b^+}\Delta u\big\} &\textup{in }\Omega\times(0,+\infty),\\
    \displaystyle u=0&\textup{on }\partial\Omega\times(0,+\infty),\\
    \displaystyle u(\cdot,0)=u_0&\textup{in }\Omega,
    \end{cases}
\end{equation}
from where we get that the limit as $b^+ \to +\infty $ solves
\begin{equation} \label{eq.main.intro.proj-1.00.88}
        \begin{cases}
        \displaystyle  b^-\partial_t \overline{u} =\min\{ \Delta \overline{u},0\} & \textup{in }\Omega \times (0,+\infty),\\
        \displaystyle \overline{u}  = 0&  \textup{on }\partial \Omega \times (0,+\infty),\\
        \displaystyle  \overline{u} (\cdot,0)  = u_0&  \textup{in } \Omega.
        \end{cases}
    \end{equation}

    In addition, we have that the function $w(x,t)=u_{b^-,b^+}(x,b^+t)$ is in fact a solution to our problem, $w=u_{\widetilde b^-,1}$, where  $\widetilde b^-=b^-/b^+$. Apply the previous results to obtain that $w(x,t)\to \bar{u} (x,t)$ as $b^+ \to +\infty$. We finally observe that we have $u(x,t)=w(x,t/b^+)\to \widetilde u(x,0)=\widetilde u_0(x)$.
\end{proof}

\begin{proof}[Proof of Theorem \ref{teo.lll}]
The proof follows from the same arguments used to prove Theorem~\ref{teo.layer}.
\end{proof}

%%%%%%%%%%%%%%%%%%%%%%%%%%%%%%%%%%%%%%%%%%%%%%%%%%%%%%%%%%
\section{Large-time behaviour} \label{sect-asymp}
\setcounter{equation}{0}

This section is devoted to study the behaviour of the solutions as $t\to\infty$. We prove Theorems \ref{u0>0}--\ref{th:bad-u0}.
\begin{proof}[Proof of Theorem \ref{u0>0}]
    Given $t>0$, let $$c(t):=\inf \Big\{k: u(x,t)\le ke^{-\lambda_1 t}\varphi(x)\textup{ for all }x\in\Omega\Big\}.$$ The estimates~\eqref{eq:starting.point.best.fitting} guarantee that this function is well defined, and comparison that it is non-increasing. Since it is also bounded from below, it has a limit, $c_*=\lim\limits_{t\to\infty} c(t)$.

    Let $v^{\tau}(x,t)=e^{\lambda_1(1+\tau)}u(x,t+\tau)$. Then, $v^\tau$ satisfies the equation and the boundary condition in~\eqref{eq.main.intro}. Moreover,
    \[
        ke^{-\lambda_1(t-1)}\varphi(x)\le v^\tau(x,t)\le c(t+\tau) e^{-\lambda_1(t-1)}\varphi(x).
    \]
    The regularity estimates from~\cite{Evans-Lenhart} ensure the existence of a subsequence along which $v^\tau$ and its derivatives up to order two converge uniformly in $\Omega$. Let $V$ be the limit of $v^\tau$ along this subsequence. On the one hand, $V$ satisfies the equation and the boundary condition in~\eqref{eq.main.intro}. In particular, it is a subsolution of the heat equation with parameter $\kappa=b^-$.  On the other hand,
    \[
        V(x,t)\le c_* e^{-\lambda_1(t-1)}\varphi(x).
    \]
    We claim that $V(x,t)\equiv c_* e^{-\lambda_1(t-1)}\varphi(x)$. Indeed, since $V$ is a subsolution and the right-hand side a solution to the heat equation $b^-\partial_t u=\Delta u$, by the strong maximum principle they cannot touch in the interior unless they are identical; and by Hopf's lemma, their normal derivatives at the boundary are strictly ordered. Therefore, $V(x,t)\le (c_*-2\varepsilon)e^{-\lambda_1(t-1)}\varphi(x)$ for some $\varepsilon>0$ small. Thus, there is some large $\tau_n$ in the subsequence such that
    \[
       e^{\lambda_1(1+\tau_n)}u(x,t+\tau_n)= v^{\tau_n}(x,t)\le (c_*-\varepsilon)e^{-\lambda_1(t-1)}\varphi(x),
    \]
    whence
    \[
        u(x,t+\tau_n)\le (c_*-\varepsilon)e^{-\lambda_1 (t+\tau_n)}\varphi(x),
    \]
    a contradiction.

    Once we have identified the limit, convergence is not restricted to a subsequence. Finally, we take $t=1$, and we get
    \[
        e^{\lambda_1(1+\tau)}u(x,1+\tau)\to c_*\varphi(x)
    \]
    as $\tau\to\infty$ uniformly in $\Omega$.
\end{proof}

\begin{proof}[Proof of Theorem \ref{u0 casi<0}]
    First, we observe that the solution to the standard heat equation becomes negative in finite time provided the projection of the initial value over the first eigenfunction is nonpositive. This can be proved by the uniform convergence to the first eigenfunction. Then apply a comparison argument and Theorem~\ref{u0>0} to conclude the result.
\end{proof}

In order to prove Theorem \ref{th:bad-u0} that deals with the bad initial data, we first define two sets of real values of $\theta>1$ and rewrite that result in a more convenient form.

Recall the notation $\theta=b^+/b^->1$, and, for instance, fix $b^->0$. Take $u_\theta=u_{b^-,\theta b^-}$, the solution corresponding to a bad initial datum, i.e., satisfying~\eqref{bad-u0}. We define the sets
\begin{equation}\label{sets-AB}
    \begin{array}{l}
        A=\Big\{\theta>1\,:\,\exists k>0 \text{ with } \displaystyle\lim_{t\to\infty} e^{\lambda_1 t} u_\theta(x,t)=k\varphi(x)\Big\}, \\ [3mm]
        B=\Big\{\theta>1\,:\,\exists k>0 \text{ with } \displaystyle\lim_{t\to\infty} e^{\lambda_2 t} u_\theta(x,t)=-k\varphi(x)\Big\}.
    \end{array}
\end{equation}

\begin{theorem}
    There exists some $\theta^*>1$ such that
    $$
        A=(1,\theta^*), \qquad B=(\theta^*,\infty).
    $$
\end{theorem}
\begin{proof}
    For any $b^->0$, as we have seen, the function $U_{b^-}(u_0)$ satisfies $U_{b^-}(u_0)(x,t)>0$ for every $x\in\Omega$ and $t>t_1$, for some $t_1$ large. Continuous dependence on $b^+$ implies that for $b^+\gtrsim b^-$, i.e. $\theta\gtrsim1$,  the same holds with $u_\theta$. Thus $\theta\in A$ and $A\supset(1,1+\varepsilon)$ for some $\varepsilon>0$. Monotonicity in $b^+$ implies that if $\theta\in A$ and $\theta'<\theta$ then $\theta'\in A$, and then $A$ is an interval $A=(1,\alpha)$.

    On the other hand, if $b^+$ is large (thus $\theta$ is large) then, again by continuous dependence, and using Theorem~\ref{th:b2-inf}, the solution $u_\theta$ is close to $\widetilde u_0$, which by definition is $\widetilde u_0(x)\le0$. Thus $\theta\in B$ and $B$ contains some interval $(M,\infty)$. The same argument could have been made taking $b^-$ small and using Theorem~\ref{th:b1-0}. Analogously as before $B$ is an interval, $B=(\beta,\infty)$.

    We want to prove that $\alpha=\beta$. To that purpose we define the sets
    $$
        \begin{array}{l}
            \displaystyle\mathcal{A}=\left\{\theta>1\,:\,\exists\, t(\theta) \text{ with } \int_\Omega u(x,t)\varphi(x)\,{\rm d}x>0 \;\forall\, t>t(\theta)\right\}, \\ [3mm]
            \displaystyle\mathcal{B}=\left\{\theta>1\,:\,\int_\Omega u(x,t)\varphi(x)\,{\rm d}x\le0 \;\forall\, t>0\right\}.
        \end{array}
    $$
    We have
    $$
        \mathcal{A}=\overline{\mathcal{B}}^c,\quad\mathcal{B}=B,\quad \mathcal{A}\subset A.
    $$
    This implies $A=\overline{B}^c$, i.e., $\alpha=\beta=\theta^*$.

    As a consequence, for $\theta=\theta^*$ the solution always changes sign for every time.
\end{proof}
We now present a one-dimensional example of the critical behavior for $\theta=\theta_*$, a solution that always changes sign and decays faster than expected. In fact, we can construct solutions that decay faster than any given exponential.

\noindent\textbf{Example.} Let $\Omega=(0,1)$. The first (normalized) eigenfunction and eigenvalue of the Laplacian in $(0,1)$ are given by $\varphi( x)=\sin(\pi x)$, $\lambda=\pi^2$. For $0<a<1$ and $k>0$ define the function
\begin{equation}\label{example-psi}
    \psi(x)=
    \begin{cases}
        -k\sin(\pi x/a)&\quad x\in (0,a), \\
        \sin(\pi(x-a)/(1-a))&\quad x\in (a,1).
    \end{cases}
\end{equation}
If we choose $k=a/(1-a)$ we would get that $\psi\in C^2(\Omega)$. Now for $\omega>0$ take
$$
    u(x,t)=e^{-\omega t}\psi(x).
$$
Since $\partial_t u=-\omega u$, we get that $\partial_t u>0$ in $(0,a)$ and $\partial_t u<0$ in $(a,1)$. Thus, in order for $u$  to be a solution to our problem we must have
$$
    \begin{cases}
        -b^+\omega\psi(x)=\psi''(x),&x\in(0,a),\\
        -b^-\omega\psi(x)=\psi''(x),& x\in(a,1),
    \end{cases}
$$
since the homogeneous boundary conditions hold trivially. The above conditions are satisfied if
$$
    b^+\omega=\frac{\pi^2}{a^2},\qquad b^-\omega=\frac{\pi^2}{(1-a)^2}.
$$
The quotient is $\theta=b^+/b^-=(1-a)^2/a^2$, which produces the values
\begin{equation}\label{example-aomega}
    a=\frac1{1+\sqrt\theta},\qquad \omega=\frac{\pi^2(1+1/\sqrt\theta)^2}{b^-}.
\end{equation}
It is clear that $\omega>\pi^2/b^-=\max\{\lambda_1,\lambda_2\}$, thus giving a faster decay than the one in \eqref{comportamiento}. Observe also that
$$
    \begin{array}{rl}
        \displaystyle I&\displaystyle=\int_\Omega u(x,0)\varphi(x)\,{\rm d}x\\[3mm]
        & \displaystyle =-k\int_0^a\sin(\pi x/a)\sin(\pi x)\,dx+\int_a^1\sin(\pi(x-a)/(1-a))\sin(\pi x)\,{\rm d}x \\ [3mm]
        &=\displaystyle\frac{(1-2a)\sin(a\pi)}{a(2-a)(1-a)^2(1+a)\pi},
    \end{array}
$$
and $I>0$ for $a<1/2$, which is precisely the case, since $\theta>1$. Thus $u(\cdot,0)=\psi$ is what we denoted as a bad initial datum. The same holds for every $t>0$, because $\displaystyle\int_\Omega u(x,t)\varphi(x)\,{\rm d}x=e^{-\omega t}I>0$.

\begin{figure}
  \centering
  \includegraphics[width=5cm]{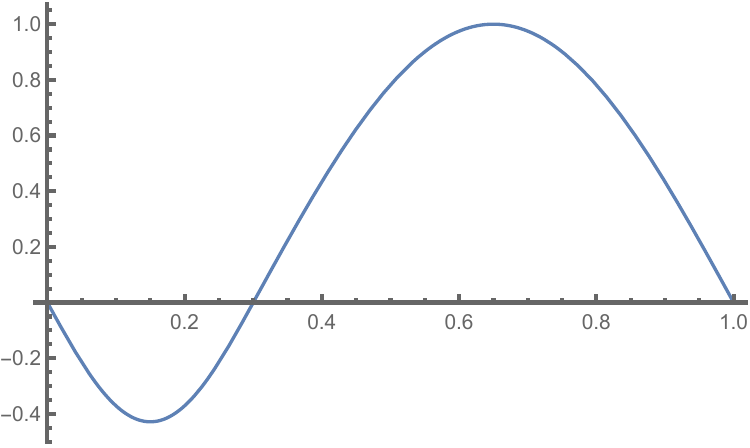}
  \caption{The initial value $\psi(x)$}
\end{figure}

On the other hand, given any $a\in(0,1/2)$, the critical value corresponding to the bad initial value $u_0=\psi$ in \eqref{example-psi} with $k=a/(1-a)$, is
$$
    \theta_*=\frac{(1-a)^2}{a^2}.
$$

Observe that $\psi'(0)=\psi'(1)$ by symmetry, so that we can glue any number of functions of the above type to generalize this example and get in that way larger decay rates. Thus, consider the function
$$
    \Psi(x)=\sum_{m=0}^\infty\psi(x-m).
$$
For any given $M\in\mathbb{N}$, the initial value function
$$
    \Psi_{M,1}(x)=\Psi(Mx)\mathds{1}_{[0,1]}(x),
$$
produces the solution
$$
    u_{M,1}(x,t)=e^{-M^2\omega t}\Psi_{M,1}(x),
$$
provided $\omega,k$ and $a$ take the values obtained previously. It has $2M-1$ sign changes, and gives a decay rate as large as we like by setting $M$ large. We complete the set of solutions in separate variables by considering the variants $u_{M,j}$ corresponding to the initial values,
$$
    \Psi_{M,2}(x)=\Psi((M+a)x), \qquad\Psi_{M,3}(x)=\Psi(Mx+a),\qquad\Psi_{M,4}(x)=\Psi((M-a)x+a),
$$
restricted to the interval $[0,1]$, see Fig.~\ref{fig2}. In $\Psi_{M,4}$ it must be $M\ge2$, otherwise it coincides with $\varphi$ and it does not change sign. Their exponent decay rates are, respectively, $(M+a)^2\omega$, $M^2\omega$ and $(M-a)^2\omega$.

\begin{figure}
  \begin{center}
  \includegraphics[width=4cm]{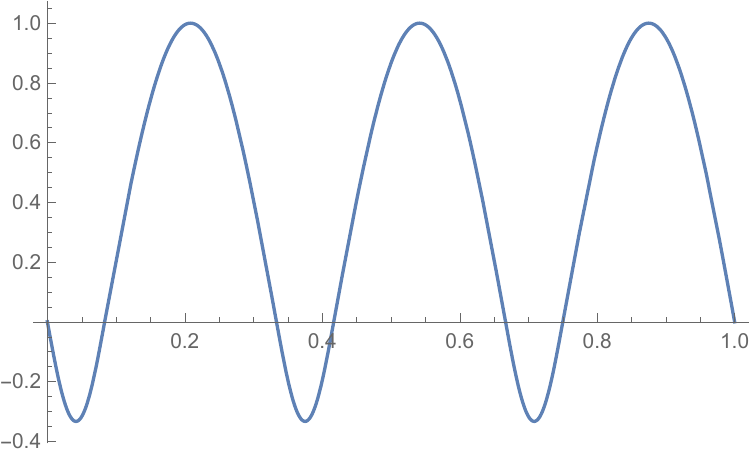}\qquad
  \includegraphics[width=4cm]{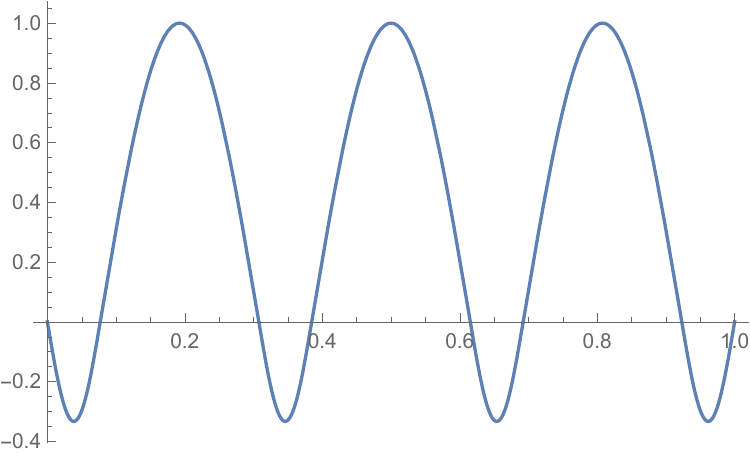} \\ [4mm]
  \includegraphics[width=4cm]{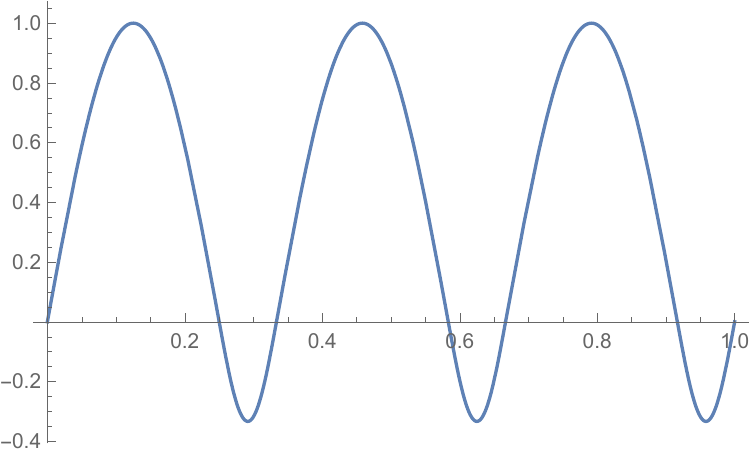}\qquad
  \includegraphics[width=4cm]{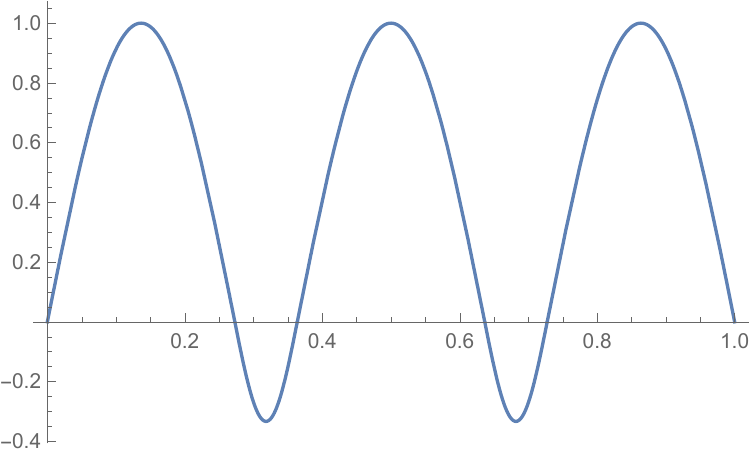}
  \end{center}
  \caption{Examples of the bad initial values $\Psi_{M,j}$ with $M=3$, $j=1,\cdots,4$.}\label{fig2}
\end{figure}

\begin{remark}
    When one tries to extend this idea to a general domain in $\mathbb{R}^N$, first one needs to show the existence of a partition of $\Omega$ into two disjoint sets $D_+$ and $D_-$ in such a way that there exist two eigenfunctions $\varphi_+$ and $\varphi_-$ of the Dirichlet Laplacian in $D_+$ and $D_-$ such that the corresponding eigenvalues are $\theta \lambda$ and $\lambda$ and the normal derivatives across the interface $\Gamma = \overline{D_1} \cap \overline{D_2} \cap \Omega$ verify $\partial_\eta \varphi_+ = - \partial_\eta \varphi_-$ on $\Gamma$.

    This partition problem is difficult to solve in general. When the domain is symmetric, like a square or an annulus, such a partition can be obtained by cutting the square into two rectangles or the annulus into two annuli.
\end{remark}

%%%%%%%%%%%%%%%%%%%%%%%%%%%%%%%%%%%%%%%%%%%%%%%
\section{Games} \label{sect-games}
\setcounter{equation}{0}

We describe in this section a dynamic game whose solution approximates a viscosity solution to Problem~\eqref{eq.main.intro}--\eqref{eq.b.intro}, taking the limit as a parameter that controls the length of the steps made in the game goes to zero.

One player (called controller)  wants to minimize the expected total payoff by choosing the best strategy according to some rules, which are as follows: starting with a token at an initial position $(x_0,t_0) \in \Omega \times (0,+\infty)$ the next position of the token after one round of the game is given by $(x_1, t_1)$ where $x_1$ is chosen at random (with uniform probability) in the ball $B_\varepsilon (x_0)$ of radius $\varepsilon$ around $x_0$, and $t_1=t_0-b \varepsilon^2$ with $b \in [C b^-,C b^+]$ being selected by the player. Here $C$ is a constant that depends only on the dimension and that we will specify later. At this new position $(x_1, t_1)$ the game is played again with the same rules. This procedure generates a sequence of positions of the token $\{(x_i, t_i)\}$. The game ends when $x_\tau \not\in \Omega$ with $t_\tau >0$, or when $t_\tau \leq 0$. The final payoff, that we denote by $\mathcal{P}$, is then given by $\mathcal{P}=0$ in the first case and $\mathcal{P}=u_0 (x_\tau)$ in the second case. A nontrivial boundary data would change the distribution of the payoff.

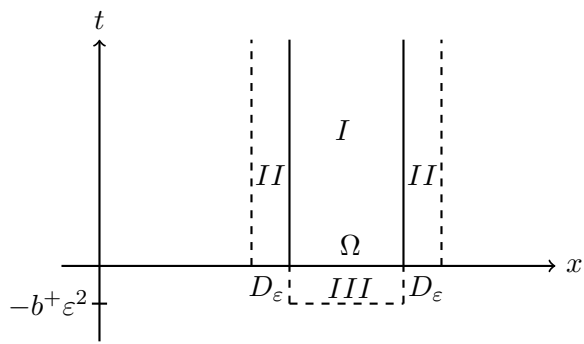
\begin{figure}[h]
    \begin{center}
    \begin{tikzpicture}[scale=1., thick]
    % Axes
    \draw[->] (-0.5,0) -- (6,0) node[right] {$x$};
    \draw[->] (0,-1) -- (0,3) node[above] {$t$};

    % Vertical dashed lines
    \draw[dashed] (2, 0) -- (2, 3);
    \draw[-] (2.5, 0) -- (2.5, 3);
    \draw[-] (4, 0) -- (4, 3);
    \draw[dashed] (4.5, 0) -- (4.5, 3);
    \draw[dashed] (2.5, -0.5) -- (2.5, 0);
    \draw[dashed] (4, -0.5) -- (4, 0);

    % Horizontal line at t = -b^+ ε²
    \draw[dashed] (2.5, -0.5) -- (4, -0.5);
    \draw[-] (-.1, -0.5) -- (.1, -0.5);

    % Region labels
    \node at (3.2, 1.8) {$I$};
    \node at (2.25, 1.2) {$II$};
    \node at (4.25, 1.2) {$II$};
    \node at (3.3, -0.3) {$III$};

    % Labels on x-axis
    \node[below] at (2.2, 0) {$D_\varepsilon$};
    \node[below] at (4.3, 0) {$D_\varepsilon$};
    \node[above] at (3.3, 0) {$\Omega$};

    % Label on t-axis
    \node[left] at (0, -0.5) {$-b^+ \varepsilon^2$};
    \end{tikzpicture}
    \end{center}
    \caption{The different regions in $(x,t)$ for the problem described above : $(I)$ where the equation holds, $(II)$ for the null exterior condition, and $(III)$ for the initial values.}\label{fig3}
\end{figure}

The value of the game is then defined as
\begin{equation} \label{value.eq.sec}
    u^\varepsilon (x,t) = \inf_{S} \mathbb{E}_S^{(x,t)} (\mathcal{P}).
\end{equation}
Here the infimum is taken among all possible strategies $S$ for the player selecting the value of $b$ at each step.

Our first goal is to show that the value function $u^\varepsilon$ defined in \eqref{value.eq.sec} is the unique solution to the maximization Problem~\eqref{eq.main.intro2}.

Let us prove in a first step that \eqref{eq.main.intro2} has a solution.

\begin{lemma}
    Given a continuous function $u_0:\overline{\Omega} \mapsto \mathbb{R}$ with $u_0\big|_{\partial \Omega} = 0$, there exists a solution to Problem \eqref{eq.main.intro2}.
\end{lemma}

\begin{proof}
We take
$$
    u_1^\eps = \min_{x\in\overline\Omega} u_0 (x).
$$
Then we define iteratively the sequence of functions $u_n^\eps$, $n\ge1$, by
\begin{equation} \label{eq.main.intro77}
    \begin{cases}
        \displaystyle u_{n+1}^\varepsilon (x,t) = \inf_{b\in [C b^-, C b^+]}
        \dashint_{B_{\varepsilon} (x)} u_n^\varepsilon (y, t - b \varepsilon^2)\,{\rm d}y,\quad
          & \Omega \times (0,+\infty), \\
        \displaystyle u_{n+1}^\varepsilon (x,t) = 0, & D_\varepsilon \times (0,+\infty), \\
        \displaystyle u_{n+1}^\varepsilon (x,t) = u_0(x), \quad & \Omega \times [-b^+\varepsilon^2,0],
    \end{cases}
\end{equation}
where $D_\varepsilon=(\Omega+B_\varepsilon)\setminus\Omega$. We have that $\{u_n\}$ is an increasing sequence, $u_{n+1} \geq u_n$,
that are subsolutions to \eqref{eq.main.intro2}, that is, they verify
\begin{equation} \label{eq.main.intro99}
    \begin{cases}
        \displaystyle u_{n+1}^\varepsilon (x,t) \leq \inf_{b\in [C b^-, C b^+]}\dashint_{B_{\varepsilon} (x)} u_{n+1}^\varepsilon (y, t - b \varepsilon^2)\,{\rm d}y,\quad& \Omega \times (0,+\infty), \\
        \displaystyle u_{n+1}^\varepsilon (x,t) = 0, & D_\varepsilon \times (0,+\infty), \\
        \displaystyle u_{n+1}^\varepsilon (x,t) = u_0(x), \quad & \Omega \times [-b^+\varepsilon^2,0].
    \end{cases}
\end{equation}
We also have that $\{u^\eps_n\}$ is a bounded sequence, since it holds that $u^\eps_n (x) \leq u_1^\eps$ for every $n\ge1$. Hence there exists the limit
$$
    \lim_{n\to \infty} u^\eps_n (x) = u^\eps (x).
$$
Passing now to the limit in the inequality in \eqref{eq.main.intro99} we obtain
\begin{equation} \label{q1}
    \displaystyle u^\varepsilon(x,t)\leq\inf_{b\in [C b^-, C b^+]}\dashint_{B_{\varepsilon}(x)}u^\varepsilon(y,t-b\varepsilon^2)\,{\rm d}y,
\end{equation}
In fact, we have that
\begin{align*}
    u_{n+1}^\varepsilon (x,t) &\leq \inf_{b\in [C b^-, C b^+]}\dashint_{B_{\varepsilon} (x)} u_{n+1}^\varepsilon (y, t - b \varepsilon^2)\,{\rm d}y \\
    &\leq \dashint_{B_{\varepsilon} (x)} u_{n+1}^\varepsilon (y, t - b \varepsilon^2)\,{\rm d}y \leq \dashint_{B_{\varepsilon} (x)} u^\varepsilon (y, t - b \varepsilon^2)\,{\rm d}y.
\end{align*}
Then, passing to the limit as $n \to \infty$, for every
$b\in [b^-,b^+]$ we have
$$
    u^\varepsilon (x,t) \leq  \dashint_{B_{\varepsilon} (x)} u^\varepsilon (y, t - b \varepsilon^2)\,{\rm d}y.
$$
Therefore, we get
$$
u^\varepsilon (x,t) \leq \inf_{b\in [C b^-, C b^+]}
\dashint_{B_{\varepsilon} (x)} u^\varepsilon (y, t - b \varepsilon^2)\,{\rm d}y.
$$

To obtain an equality here we need to further analyze
the sequence $\{u_n^\eps\}$. Let us call, to simplify the notation,
$G$ to the prescribed values, that is,
$$
    G(x,t)=
    \begin{cases}
        \displaystyle  0, & (x,t)\in D_\varepsilon \times (-C b^+ \eps^2,+\infty), \\
        \displaystyle u_0(x), \quad & (x,t)\in \Omega \times [-C b^+\varepsilon^2,0].
    \end{cases}
$$

For $(x,t) \in \Omega \times (0,C b^-\eps^2)$ we have that the number
\begin{equation}\label{inf-un}
    \inf_{b\in [C b^-,C b^+]}\dashint_{B_{\varepsilon} (x)} u_n^\varepsilon (y, t - b \varepsilon^2)\,{\rm d}y
\end{equation}
involves only negative times and therefore it does not depend on $n$.
Hence,
$$
    u_{n+1}^\varepsilon (x,t) = \inf_{b\in [C b^-,C b^+]}\dashint_{B_{\varepsilon} (x)} u_n^\varepsilon (y, t - b \varepsilon^2)\,{\rm d}y= \inf_{b\in [C b^-, C b^+]}\dashint_{B_{\varepsilon} (x)} G (y, t - b \varepsilon^2)\,{\rm d}y,
$$
for every $n\geq 1$.
Since, $u_n = G$ for every $n\geq 2$ outside $\Omega \times (0,\infty)$ we obtain that
$$
    u^\varepsilon (x,t) = \inf_{b\in [C b^-, C b^+]}\dashint_{B_{\varepsilon} (x)} u^\varepsilon (y, t - b \varepsilon^2)\,{\rm d}y,
$$
for $(x,t) \in \Omega \times (0,C b^-\eps^2)$.

Now, for $(x,t) \in \Omega \times [C b^-\eps^2, 2 C b^-\eps^2)$ we have that the above quantity \eqref{inf-un} involves only times that are smaller than $C b^-\eps^2$ and therefore from our previous computations we have that it does not depend on $n$ for $n\geq 3$. Then we get that
$$
    u_{n+1}^\varepsilon (x,t) = \inf_{b\in [C b^-,C b^+]}\dashint_{B_{\varepsilon} (x)} u^\varepsilon (y, t - b \varepsilon^2)\,{\rm d}y
$$
for every $n\geq 2$ and hence we obtain
$$
    u^\varepsilon (x,t) = \inf_{b\in [C b^-,C b^+]}\dashint_{B_{\varepsilon} (x)} u^\varepsilon (y, t - b \varepsilon^2)\,{\rm d}y,
$$
for $(x,t) \in \Omega \times [C b^-\eps^2, 2C b^-\eps^2)$.

Iterating this procedure we conclude that
\begin{equation} \label{q12}
    \displaystyle u^\varepsilon (x,t) = \inf_{b\in [C b^-,C b^+]}\dashint_{B_{\varepsilon} (x)} u^\varepsilon (y, t - b \varepsilon^2)\,{\rm d}y,
\end{equation}
for $(x,t) \in \Omega \times (0,+\infty)$.

Finally, we also have
$$
    u^\varepsilon(x,t)  = 0 \quad\textup{ if }(x,t)\in D_\varepsilon \times (0,+\infty),\qquad u^\varepsilon(x,t) = u_0(x)\quad \textup{ if }(x,t)\in  \Omega \times [-b^+\varepsilon^2,0],
$$
and we conclude that $u^\eps$ solves \eqref{eq.main.intro2}.
\end{proof}

With the help of this existence result we next show that any solution to \eqref{eq.main.intro2} is the value of the game.

First, let us introduce some concepts from probability theory.

%%%%%%%%%%%%%%%%%%%%%%%%%%%%%%%%%%%%%5
\subsection{Probability. The Optional Stopping Theorem.}
We briefly recall (see \cite{Williams}) that a sequence of random variables $\{M_{k}\}_{k\geq 1}$ is a supermartingale (submartingale) if
$$
    \mathbb{E}[M_{k+1}\arrowvert M_{0},M_{1},\dots,M_{k}]\leq M_{k} \ \ (\geq).
$$
We have the following result, known as the Optional Stopping Theorem (denoted by {\it OST} in what follows).

\begin{theorem} \label{teo.OST}
    Let $\{M_{k}\}_{k\geq 0}$ be a supermartingale and let $\tau$ be a stopping time such that one of the following conditions hold:
    \begin{itemize}
        \item[(a)] the stopping time $\tau$ is bounded almost surely;
        \item[(b)] $\mathbb{E}[\tau]<\infty$ and there exists a constant $c>0$ such that $\mathbb{E}[M_{k+1}-M_{k}\arrowvert M_{0},...,M_{k}]\leq c$;
        \item[(c)] there exists a constant $c>0$ such that $|M_{\min \{\tau,k\}}|\leq c$ almost surely for every $k$.
    \end{itemize}
    Then
    \begin{equation}\label{supermartin}
        \mathbb{E}[M_{\tau}]\leq \mathbb{E} [M_{0}].
    \end{equation}
    The inequality in \eqref{supermartin} is reversed if $\{M_{k}\}_{k\geq 1}$ is a submartingale.
\end{theorem}
For the proof of this classical result we refer to \cite{Doob,Williams}.

Now, we are ready to show that the game has a value.

\begin{lemma}
    Let $u^\eps$ be a solution to \eqref{eq.main.intro2}. Then $u^\eps$ coincides with the value of the game, that is, we have that
    \begin{equation} \label{value.eq.sec2}
        u^\varepsilon (x,t) = \inf_{S} \mathbb{E}_S^{(x,t)} (\mathcal{P}),
    \end{equation}
where $\mathcal{P}=\{\mbox{final payoff}\}$ is described above.
\end{lemma}

\begin{proof} Let us design an strategy for the player following the solution $u^\eps(x,t)$ to \eqref{eq.main.intro2}. At each position of the game $(x_k,t_k)\in  \Omega \times (0,T)$, the player will choose $\widetilde{b_k}=S(x_k,t_k)$ according to
$$
    \begin{aligned}
        \dashint_{B_{\varepsilon} (x)} u^\varepsilon (y, t - \widetilde{b_k} \varepsilon^2)\,{\rm d}y&\geq \inf_{b\in [C b^-,C b^+]}
        \dashint_{B_{\varepsilon} (x)} u^\varepsilon (y, t - b \varepsilon^2)\,{\rm d}y \\[10pt]
        &\geq\dashint_{B_{\varepsilon} (x)} u^\varepsilon (y, t - \widetilde{b_k} \varepsilon^2)\,{\rm d}y - \frac{\eta}{2^{k+1}}.
    \end{aligned}
$$
Assume that the player always plays according to this strategy (that we call $S^*$). Define the sequence
$$
    M_k = u^\eps (x_k,t_k) + \frac{\eta}{2^{k}}
$$
and let us see that $\{M_k\}$ is a supermartingale. In fact, we have
$$
    \begin{aligned}
        \mathbb{E}(M_{k+1}|(x_k,t_k))&=\mathbb{E}(u^\eps(x_{k+1},t_{k+1})+\frac{\eta}{2^{k+1}}|(x_k,t_k))\\[10pt]
        &=\mathbb{E}(u^\varepsilon (x_{k+1}, t_{k} - \widetilde{b_{k}} \varepsilon^2)\,{\rm d}y + \frac{\eta}{2^{k+1}} |(x_k,t_k)) \\[10pt]
        &=\dashint_{B_{\varepsilon} (x_{k})} u^\varepsilon (y, t_{k} - \widetilde{b_k} \varepsilon^2)\,{\rm d}y + \frac{\eta}{2^{k+1}}  \\[10pt]
        &\leq\inf_{b\in[b^-,b^+]}\dashint_{B_{\varepsilon}(x_k)}u^\varepsilon(y,t_k-b\varepsilon^2)\,{\rm d}y+\frac{\eta}{2^{k+1}}+\frac{\eta}{2^{k+1}}\\[10pt]
        &= u^\varepsilon (x_k, t_k) + \frac{\eta}{2^{k}} = M_k.
    \end{aligned}
$$

According to the OST, we have
$$
    \mathbb{E}_{S^*} (\mathcal{P})  \geq \mathbb{E} (M_\tau) = \lim_{k\to \infty} \mathbb{E} (M_{\tau \wedge k})
    \geq M_0 = u^\eps (x_0,t_0) + \eta.
$$
Then we conclude that the value of the game verifies
$$
    \sup_{S}  \mathbb{E}_{S_1}^{(x_0,t_0)} (\mathcal{P}) \leq \mathbb{E}_{S^*}^{(x_0,t_0)} (\mathcal{P}) \leq u^{\eps } (x_0,t_0) + \eta.
$$

On the other hand, if we now consider the sequence
$$
N_k =u^\eps (x_k,t_k)
$$
we have that $\{N_k\}$ is a submartingale,
$$
    \begin{aligned}
        \mathbb{E} (N_{k+1} |(x_k,t_k))&=\mathbb{E} (u^\eps (x_{k+1},t_{k+1}) |(x_k,t_k))\\[10pt]
        &=\mathbb{E}(u^\varepsilon (x_{k+1}, t_{k} - \widetilde{b_{k}} \varepsilon^2)\,{\rm d}y  |(x_k,t_k))\\[10pt]
        &=\dashint_{B_{\varepsilon} (x_{k})} u^\varepsilon (y, t_{k} - \widetilde{b_k} \varepsilon^2)\,{\rm d}y\\[10pt]
        &\leq\inf_{b\in [C b^-,C b^+]}\dashint_{B_{\varepsilon}(x_k)}u^\varepsilon (y, t_k - b \varepsilon^2)\,{\rm d}y\\[10pt]
        &= u^\varepsilon (x_k, t_k)= N_k.
    \end{aligned}
$$

According again to the OST, we have
$$
    \mathbb{E}_{S^*} (\mathcal{P})  \geq \mathbb{E} (N_\tau) = \lim_{k\to \infty} \mathbb{E} (N_{\tau \wedge k})\geq N_0 = u^\eps (x_0,t_0).
$$
Then, we conclude that the value of the game verifies
$$
    \inf_{S}  \mathbb{E}_{S}^{(x_0,t_0)} (\mathcal{P}) \geq  u^{\eps } (x_0,t_0) .
$$
Since, $\eta>0$ is arbitrary, we conclude that
$$
    \inf_{S}  \mathbb{E}_{S}^{(x_0,t_0)} (\mathcal{P}) =  u^{\eps } (x_0,t_0)
$$
as we wanted to show.
\end{proof}

Now, we want to study the limit of $u^\eps$ as $\eps \to 0$.
To obtain a convergent subsequence we will use the following
Arzela-Ascoli type lemma. For its proof see Lemma~4.2 from \cite{MPRb}.

\begin{lemma}\label{lem.ascoli.arzela}
    Let $\{u^\eps : \overline{\Omega}\times [0,T]\to \R,\ \eps>0\}$ be a set of functions such that:
    \begin{enumerate}
        \item[\rm (i)] there exists $C>0$ such that $|u^\eps (x,t)|<C$ for every $\eps >0$ and every $(x,t)\in\overline{\Omega}\times [0,T]$;
        \item[\rm (ii)]\label{cond:2} given $\delta >0$ there are positive constants $r_0$ and $\eps_0$ such that for every $\eps<\eps_0$ and any $x, y \in \overline{\Omega}$ with $|x - y | < r_0 $, and any $s,t \in [0,T]$ with $|s - t | < r_0 $, it holds
            $$
                |u^\eps (x,t) - u^\eps (y,s)| < \delta.
            $$
    \end{enumerate}
    Then there exists a uniformly continuous function $u:\overline{\Omega}\times[0,T]\to\R$, and a subsequence, still denoted by~$\{u^\eps\}$, such that
    $$
        u^{\eps}\to u \quad \text{ as } \eps\to 0, \quad \textrm{ uniformly in } \overline{\Omega} \times[0,T].
    $$
\end{lemma}
Therefore, we must show now that the family $\{u^\eps\}$ of solutions to Problem \eqref{eq.main.intro2} verifies the hypothesis of the previous lemma. To this end we first observe that our proof of existence for that problem shows that $\{u^\eps\}$ is uniformly bounded.

\begin{lemma}\label{lema.bound}
    In the above notation it holds that $\|u^\eps\|_\infty \leq \|u_0\|_\infty$.
\end{lemma}

To show the second condition of Lemma \ref{lem.ascoli.arzela} we need some estimates for the random walk.
%%%%%%%%%%%%%%%%%%%%%%%%%%%%%%%%%%%%
\subsection{Estimates for the Random Walk game}

Now we assume that we are playing with the random walk game inside $\Omega$ with a stopping rule that ends the game the first time exiting $\Omega$.

In the next result we show that when we start playing this random walk game sufficiently close to a boundary point, then we exit the domain close to that point with high probability and also the number of rounds needed to exit is relatively small with high probability (when we look at the
appropriate scale that in this case is of order $1/\varepsilon^2$).

\begin{lemma} \label{lema.4.5.}
    Given $\eta>0$ and $a>0$, there exists $r_{0}>0$ and $\eps_{0}>0$ such that, given $y\in \partial\Omega$ and $x_{0}\in\Omega$ with $|x_{0}-y|<r_{0}$, if we play random in $\Omega$ (choosing the next position of the game uniformly in $B_\varepsilon (x)$) we obtain
    $$
        \mathbb{P} \Big(|x_{\tau}-y|< a \Big) \geq 1 - \eta
    $$
    for every $\eps<\eps_{0}$ and $x_{\tau}$ the first position outside $\Omega$.

    Moreover, if we count the number of rounds, $\tau$, we can also get
    $$
        \mathbb{P} \Big(\tau \geq \frac{a}{2\eps^2} \Big)< \eta.
    $$
\end{lemma}

\begin{proof}
We include only a sketch of the proof. Assume that $N\geq3$ (the cases $N=1,2$ are similar). The first step is, given $\theta<\theta_{0}$, and $y\in\Omega$, we are going to assume that we have $\overline{B_{\theta}(0)}\cap\overline{\Omega}=\{y\}$ (since $\Omega$ is smooth it satisfies a uniform exterior ball condition). We define the set $\Omega_{\eps}=\{x\in\R^{N}:d(x,\Omega)<\eps\}$ for $\eps$ small enough. Now, we consider the function $\mu:\Omega_{\eps}\rightarrow\R$ given by
\begin{equation}\label{mu}
    \mu(x)=\frac{1}{\theta^{N-2}}-\frac{1}{|x |^{N-2}}.
\end{equation}
This function is positive in $\overline{\Omega}\backslash\{y\}$, radially increasing and harmonic in $\Omega$. It also holds that $\mu(y)=0$.

Take the first position of the game, $x_{0}\in\Omega $, such that $|x_{0}-y|<r_{0}$ with $r_{0}$ to be chosen later. Let $ \{x_{k}\}_{k \geq 0} $ be the sequence of positions of the game playing random walk and consider the sequence of random variables
$$
    N_{k}=\mu(x_{k})
$$
for $k\geq 0$. Since $\mu$ is harmonic, we have that $\{N_{k}\}$ is a martingale,
$$
    \mathbb{E} [N_{k + 1} \arrowvert N_{k}] = \dashint_{B _{\eps} (x_{k})} \mu(y) dy = \mu (x_{k}) = N_{k}.
$$
Since $ \mu $ is bounded in $\Omega$, the third hypothesis of the OST is fulfilled, hence we obtain
\begin{equation}\label{muxo}
    \mathbb{E} [\mu(x_{\tau})] = \mu(x_ {0}).
\end{equation}
We have the following estimate for $\mu(x_{0})$: there exists a constant $c(\Omega,\theta)>0$ such that
$$
    \mu(x_0)\leq c(\Omega,\theta)r_0.
$$
Now we need to establish a relation between $\mu(x_{\tau})$ and $|x_{\tau}-y|$. To this end, we take the function $b:[\theta,+\infty)\rightarrow\R$ given by
\begin{equation}\label{funb}
    b(\overline{a})=\frac{1}{\theta^{N-2}}-\frac{1}{\overline{a}^{N-2}}.
\end{equation}
Note that this function is the radial version of $ \mu $. It is positive and increasing, so it has an inverse (also increasing) that is given by the formula
$$
    \overline{a}(b)=\frac{\theta}{(1-\theta^{N-2}b)^{\frac{1}{N-2}}}.
$$
With this function we can get the following result: given $a>0$, there exist $\overline{a}>\theta$, $b>0$ and $\eps_{0}>0$ such that
$$
    \mu(x_{\tau})<b \;\Longrightarrow\; |x_{\tau}-y|<a \ , \ d(x_{\tau},\Omega)<\eps_{0}.
$$
Then we have $\mathbb{P}(\mu(x_{\tau})\geq b)\geq \mathbb{P}(|x_{\tau}-y|\geq a)$, and we obtain
\begin{equation}\label{desnorma}
    \mathbb{P}(|x_{\tau}-y|\geq a)<\eta
\end{equation}
if $r_0$ is small. Now let us compute
\begin{equation}\label{ENK}
    \mathbb{E}[N_{k+1}^{2}-N_{k}^{2}\arrowvert N_{k}]=
    \dashint_{B_{\eps}(x_{k})}(\mu^{2}(w)-\mu^{2}(x_{k})) \,{\rm d}w.
\end{equation}
Using a Taylor expansion of order two we can prove that
$$
    \mathbb{E}[N_{k+1}^{2}-N_{k}^{2}\arrowvert N_{k}]\geq \sigma(\Omega)\eps^{2}.
$$
Then, with arguments similar to those used above, we obtain
$$
    \mathbb{P} \Big(\tau\geq\frac{a}{2\eps^{2}}\Big)<\eta
$$
for $r_{0}$ small enough.
\end{proof}

Now we are ready to prove the second condition in the Arzela-Ascoli type lemma.

\begin{lemma}\label{lem.ascoli.arzela.asymp} Given $\delta>0$ there are
	$r_0>0$ and $\eps_0>0$ such that for every $0< \eps < \eps_0$
	and any $x, y \in \overline{\Omega}$ with $|x - y | < r_0 $ and $|t-s|< r_0$
	it holds
	$$
	|u^\eps (x,t) - u^\eps (y,s)| < \delta.
	$$
\end{lemma}

\begin{proof}
Recall that $u^{\eps}$ is the value of the game. Let $\Omega_T=\Omega\times[0,T]$ and consider also the parabolic boundary $\partial_p \Omega_T=(\Omega\times\{0\})\cup(\partial\Omega\times[0,T])$.
	
We start with two close points $(x,t)$ and $(y,s)$. Let $G:[(\R^N\backslash\Omega\times (0,T])\cup (\R^N\times\{0\})]\rightarrow\R$ be given by
$$
    G(x,t) =
    \begin{cases}
        0 &\mbox{if }t\geq -1,\ x\notin\Omega,  \\
        \displaystyle u_0(x) &\mbox{if }t<0,\ x\in\Omega.
    \end{cases}
$$
From our conditions on the data, the function $G$ is well defined and is uniformly continuous in both variables, that is,
$$
    |G(x,t)-G(y,s)|\leq L(|x-y|+|t-s|),
$$
for some uniform modulus of continuity $L(\cdot)$. Therefore, if $(y,s)\in \partial_p \Omega_T$ and $(x,t)\in\partial_p \Omega_T$ we get that
$$
    |u^\eps (x,t) - u^\eps (y,s)|= |G (x) - G (y)| < \delta,
$$
when $|x-y| <r_0$.

Now, we want to obtain a similar bound when $(y,s)\in \partial_p \Omega_T$ and $(x,t)\in\Omega \times (0,T)$ are close.

Fix a strategy $S$ for the player. Given $\eta >0$ and $a>0$, we have $r_{0}$ and $\eps_{0}$ from Lemma~\ref{lema.4.5.}. Let us consider the event
$$
    F=\Big\{ \tau < \lceil \frac{a}{2\eps^{2}}\rceil \Big\}.
$$

We consider two cases.
	
\noindent\textbf{1st case:} We are going to show that
$$
    u^{\eps}(x_{0},t_0)-G(y,s) \geq  - A(a,\eta),
$$
with $A(a,\eta)\searrow 0$ if $a\rightarrow 0$ and $\eta\rightarrow 0$. We have
$$
    u^{\eps}(x_{0},t_0) = \inf_{S}\mathbb{E}^{(x_{0},t_0)}_{S}[G(x_{\tau},t_{\tau})].
$$
Now, for any fixed strategy $S$, we obtain
$$
    \begin{aligned}
        \displaystyle\mathbb{E}^{(x_{0},t_0)}_{S}[G (x_{\tau},t_{\tau})]&\displaystyle=\mathbb{E}^{(x_{0},t_0)}_{S}[G(x_{\tau},t_{\tau})\arrowvert F]\mathbb{P}(F)+\mathbb{E}^{(x_{0},t_0)}_{S}[G(x_{\tau},t_{\tau})\arrowvert F^{c}]\mathbb{P}(F^{c})\\[10pt]
        & \displaystyle \geq \mathbb{E}^{(x_{0},t_0)}_{S}[G(x_{\tau},t_{\tau})\arrowvert F]\mathbb{P}(F) - \max\{|G|\}\mathbb{P}(F^{c}).
    \end{aligned}
$$
We have an estimate for $\mathbb{P}(F^{c})$,
$$
    \mathbb{P}(F^{c})\leq \mathbb{P} \Big(\tau\geq\lceil \frac{a}{2\eps^{2}}\rceil \Big).
$$
In fact, we observe that using Lemma \ref{lema.4.5.} we get
\begin{equation}\label{Ac2}
    \mathbb{P}\Big(\tau \geq \frac{a}{2\eps^{2}}\Big) \leq \mathbb{P}\Big(\tau \geq \frac{a}{2\eps_{0}^{2}}\Big)< \eta,
\end{equation}
for $\eps < \eps_{0}$, that is,
$$
    \mathbb{P}(F^{c})\leq \eta
$$
and hence
$$
    \mathbb{P}(F) =1-\mathbb{P}(F^{c}) \geq 1-\eta.
$$
Then we obtain
\begin{equation}\label{arriba}
    \displaystyle  \mathbb{E}^{(x_{0},t_0)}_{S}[G(x_{\tau},t_{\tau})]\geq \mathbb{E}^{(x_{0},t_0)}_{S}[G(x_{\tau},t_{\tau})\arrowvert F] (1-\eta)
    - \max\{lG|\}\eta.
\end{equation}

Let us analyze the expected value $\mathbb{E}^{(x_{0},t_0)}_{S}[G(x_{\tau},t_{\tau})\arrowvert F]$. Again we need to consider two events,
$$
    F_{1}=F\cap [\{ |x_{\tau}-y|< a\}\cap\{ |t_{\tau}-s|<a\}] \quad \mbox{and } F_{2}=F\cap F_{1}^{c}.
$$
We have that $ F=F_{1}\cup F_{2}$. Then
\begin{equation}\label{retomo}
    \begin{aligned}
        \displaystyle\mathbb{E}^{(x_{0},t_0)}_{S}[G(x_{\tau},t_{\tau})\arrowvert F]&=\mathbb{E}^{(x_{0},t_0)}_{S}[G(x_{\tau},t_{\tau})\arrowvert F_{1}]\mathbb{P}(F_{1}) \\[10pt]
        &\quad \displaystyle +\mathbb{E}^{(x_{0},t_0)}_{S}[G(x_{\tau},t_{\tau})\arrowvert F_{2}]\mathbb{P}(F_{2}).
    \end{aligned}
\end{equation}
Now we have that
\begin{equation}\label{a2}
    \begin{aligned}
        \mathbb{P}(F_{2})&\leq \mathbb{P}([\{ |x_{\tau}-y|< a\}\cap\{ |t_{\tau}-s|<a\}]^c)\\[10pt]
        &= \mathbb{P}(\{ |x_{\tau}-y|\geq a\}\cup \{ |t_{\tau}-s|\geq a\}) \\[10pt]
        &\leq \mathbb{P}(|x_{\tau}-y|\geq a)+\mathbb{P}(|t_{\tau}-s|\geq a)<2\eta.
    \end{aligned}
\end{equation}

To get a bound for the other case we observe that
$$
    F_{1}^{c}=F^{c}\cup\{ |x_{\tau}-y|\geq a\}\cup \{|t_{\tau}-s|\geq a\},
$$
and therefore
$$
    \mathbb{P}(F_{1})=1-\mathbb{P}(F_{1}^{c})\geq 1-[\mathbb{P}(F^{c})+\mathbb{P}(|x_{\tau}-y|\geq a)+\mathbb{P}(|t_{\tau}-s|\geq a)],
$$
so we arrive to
\begin{equation}\label{a1}
    \mathbb{P}(F_{1})\geq 1-3\eta.
\end{equation}

If we go back to \eqref{retomo} and use \eqref{a1} and \eqref{a2} we get
\begin{equation}\label{retomo2}
    \mathbb{E}^{x_{0}}_{S}[G(x_{\tau},t_{\tau})\arrowvert F]\geq\mathbb{E}^{x_{0}}_{S}[G(x_{\tau},t_{\tau})\arrowvert F_{1}] (1-3\eta)- \max \{|G|\}2\eta .
\end{equation}
Using that $G$ is uniformly continuous (recall that we assume that $u_0$ is continuous with $u_0|_{\partial \Omega}=0$) we obtain
$$
    G(x_{\tau},t_{\tau})\geq G(y,s)-L(|x_{\tau}-y|+|t_{\tau}-s|)\geq G(y,s)-L(2a),
$$
and then, using that $(G(y,s) - L(2a))$ does not depend on the strategies, we conclude that
\begin{equation*}
    \mathbb{E}^{x_{0}}_{S}[G(x_{\tau},t_{\tau})\arrowvert F]\geq G(y,s) - L(2a)  - \max \{lG|\}5\eta.
\end{equation*}
Recalling \eqref{arriba} we obtain
$$
    \mathbb{E}^{x_{0}}_{S}[G ( x_{\tau},t_{\tau})]  \geq G(y,s)- L(2a)  - \max \{lG|\}5\eta.
$$
Notice that when $\eta \rightarrow 0$ and $a\rightarrow 0$ the the right hand side goes to $G(y,s)$, hence we have obtained
$$
    \mathbb{E}^{x_{0}}_{S}[G( x_{\tau},t_{\tau})]\geq G(y,s) - A(a,\eta).
$$
Taking the infimum over all possible strategies $S$ we get
$$
    u^{\eps}(x_{0},t_0)\geq G(y,s) - A(a,\eta)
$$
with $A(a,\eta)\to 0$ as $\eta\rightarrow 0$ and $a\rightarrow 0$ as we wanted to show.

\noindent\textbf{2nd case:} Now we want to show that $$u^{\eps}(x_{0},t_0)-G(y,s)\leq  B(a,\eta),$$ with $B(a,\eta)\searrow 0$ as $\eta\rightarrow 0$ and $a\rightarrow 0$. In this case we just use the strategy $S^*$ given by $b=C b^-$ at every step as the strategy for the player,
$$
    u^{\eps}(x_{0},t_0)\leq \mathbb{E}^{x_{0},t_0}_{S^{*}}[G(x_{\tau},t_{\tau})].
$$
Using again the set $F$ that we considered in the previous case we obtain
$$
    \mathbb{E}^{x_{0},t_0}_{S^{*}}[ G( x_{\tau},t_{\tau})]= \mathbb{E}^{x_{0},t_0}_{S^{*}}[G( x_{\tau},t_{\tau})\arrowvert F]\mathbb{P}(F)+
    \mathbb{E}^{x_{0}}_{S^{*}}[G( x_{\tau},t_{\tau})\arrowvert F^{c}]\mathbb{P}(F^{c}).
$$
We have that $\mathbb{P}(F) \leq 1$ and $\mathbb{P}(F^{c})\leq 2\eta $. Hence we get
\begin{equation}\label{retomo3}
    \mathbb{E}^{x_{0},t_0}_{S^{*}}[G( x_{\tau},t_{\tau})]\leq \mathbb{E}^{x_{0},t_0}_{S^{*}}[G( x_{\tau},t_{\tau})\arrowvert F]+\max\{|G|\} 2\eta.
\end{equation}
To bound $\mathbb{E}^{x_{0},t_0}_{S^{*}}[G(x_{\tau},t_{\tau})\arrowvert F]$ we will use again the sets $F_{1}$ and $F_{2}$
as in the previous case. We have
$$
    \mathbb{E}^{x_{0},t_0}_{S^{*}}[G(x_{\tau},t_{\tau})\arrowvert F]=\mathbb{E}^{x_{0},t_0}_{S^{*}}[G(x_{\tau},t_{\tau})\arrowvert F_{1}]\mathbb{P}(F_{1})+\mathbb{E}^{(x_{0},t_0)}_{S^{*}}[G(x_{\tau},t_{\tau})\arrowvert F_{2}]\mathbb{P}(F_{2}).
$$
Now we use that $\mathbb{P}(F_{1}) \leq 1$ and $\mathbb{P}(F_{2})\leq c\eta$ to obtain
$$
    \mathbb{E}^{(x_{0},t_0)}_{S^{*}}[G(x_{\tau},t_{\tau})\arrowvert F]\leq \mathbb{E}^{x_{0},t_0}_{S^{*}}[G(x_{\tau},t_{\tau})\arrowvert F_{1}]+ \max\{|G|\}2\eta.
$$
Using, as before, that $G$ is Lipschitz and that $G(y,s) + L(2a)$ does not depend on the strategies we get
$$
    \begin{array}{rl}\displaystyle\mathbb{E}^{x_{0},t_0}_{S^{*}}[G(x_{\tau},t_{\tau})\arrowvert F]
    &\displaystyle\leq \mathbb{E}^{x_{0},t_0}_{S^{*}}[G(y,s) + L(2a)\arrowvert F_{1}]+ \max\{|G|\}2\eta \\ [2mm]
    &\displaystyle\leq G(y,s) +L(2a) +\max\{lG|\}2\eta ,
    \end{array}
$$
and therefore we conclude that
$$
    \mathbb{E}^{x_{0}t_0}_{S^{*}}[G( x_{\tau},t_{\tau} )]\leq G(y,s) + L(2a) + \max\{G|\}[2\eta].
$$
We have proved that
$$
    \mathbb{E}^{x_{0},t_0}_{S^{*}}[ G( x_{\tau},t_{\tau})]\leq G(y,s) + B(a,\eta)
$$
with $B(a,\eta)\to 0$. Taking infimum over the strategies for the player
$$
    u^{\eps}(x_{0},t_0)\leq G(y,s) + B(a,\eta)
$$
with $B(a,\eta)\rightarrow 0$ as $\eta\rightarrow 0$ and $a\rightarrow 0$.

We conclude that
$$
    |u^{\eps}(x_{0},t_0)-G(y,s)|< \max\{A(a,\eta),B(a,\eta)\} \to 0,
$$
as $\eta\rightarrow 0$ and $a\rightarrow 0$, that holds when $(y,s)\in\partial_p \Omega_T$ and $(x_0,t_0)\in\Omega_T$ is close to $(y,s)$.

Now, given two points $(x_0,t_0)$ and $(z_0,s_0)$ inside $\Omega$ with $|x_0-z_0|<r_0$ and $|t_0-s_0|<r_0$ we couple the game starting at $(x_0,t_0)$ with the game starting at $(z_0,s_0)$ making the same movements at every play. This coupling generates two sequences of positions $\{(x_i,t_i)\}$ and $\{(z_i,s_i)\}$ such that $|x_i - z_i|<r_0$, $|t_i - s_i|<r_0$. This continues until one of the games exits the domain (say at $(z_\tau,s_{\tau}) \not\in \Omega\times (0,T)$). At this point for the game starting at $(x_0,t_0)$ we have that its position $(x_\tau,t_{\tau})$ is close to the exterior point $(z_\tau,s_{\tau}) \not\in \Omega\times (0,T)$ (since we have $|x_\tau - z_\tau|<r_0$ and $|t_\tau - s_\tau|<r_0$) and hence we can use our previous estimates for points close to the boundary to conclude that
\[
    |u^{\eps}(x_{0},t_0)- u^\eps (z_0,s_0)|< \delta. \qedhere
\]
\end{proof}

Now, we are ready to show that the sequence of solutions $u^\eps$ of the games converges uniformly in $\overline{\Omega} \times [0,T]$, and that the limit $u$ is the unique viscosity solution to Problem~\eqref{eq.main.intro}--\eqref{eq.b.intro}.

\begin{proof}[Proof of Theorem \ref{teo.juegos.intro}]
The uniform convergence $$\lim\limits_{\varepsilon\to0}u^\varepsilon=u,$$ up to a subsequence, follows from the Arzela-Ascoli type lemma. From the fact that $u^\eps (x,t) = G(x)$ for $(x,t) \in \partial_p \Omega_T$ we obtain that
$$
    u(x,t) = 0, \qquad x\in \partial \Omega, t>0,
$$
and also we have that
$$
    u(x,0) = u_0 (x), \qquad x \in \Omega.
$$

Let us show now first that $u$ is a viscosity subsolution to the equation in \eqref{eq.main.intro}. To this end, take a smooth function $\varphi\colon\Omega\times[0,\infty)\to \mathbb{R}$ such that $u-\varphi$ has a strict maximum in $B_\delta(x_0)\times (t_0-\delta,t_0+\delta) \subseteq \Omega\times (0,T)$ for some $\delta>0$ at $(x_0,t_0)$. Then, from the uniform convergence of $u^\eps$ to $u$ we have that $u^\eps-\varphi$ has almost a maximum in $B_\delta(x_0)\times (t_0-\delta,t_0+\delta)\subseteq \Omega\times (0,T)$ at $(x_\eps,t_\eps)$, that is, we have that
$$
    u^\eps(y,s)-\varphi (y,s)\leq u^\eps (x_\eps,t_\eps)-\varphi(x_\eps,t_\eps) +\eps^3,
$$
with
$$
    (x_\eps,t_\eps) \to (x_0,t_0) \quad \mbox{ as } \eps \to 0.
$$
Now, let us rewrite the previous inequality as
$$
    u^\eps(y,s)-u^\eps (x_\eps,t_\eps) \leq\varphi (y,s)-\varphi(x_\eps,t_\eps) +\eps^3,
$$
and use the Dynamic Programing Principle (DPP for short) at $(x_\eps,t_\eps)$,
\begin{equation}\label{q18}
    \displaystyle 0= \inf_{b\in [C b^-, C b^+]}
    \dashint_{B_{\varepsilon} (x_\eps)} \left(u^\varepsilon (y, t_\eps - b \varepsilon^2)- u^\varepsilon (x_\eps,t_\eps)\right) \,{\rm d}y,
\end{equation}
to obtain
\begin{equation}\label{q189}
    \displaystyle 0 \leq \inf_{b\in [C b^-, C b^+]}
    \dashint_{B_{\varepsilon} (x_\eps)} \left(\varphi (y, t_\eps - b \varepsilon^2)- \varphi (x_\eps,t_\eps)\right) \,{\rm d}y + o(\eps^2).
\end{equation}
That is,
\begin{equation}\label{q1890}
    \begin{array}{rl}
        0&\displaystyle  \leq \inf_{b\in [C b^-, C b^+]} \left\{\dashint_{B_{\varepsilon} (x_\eps)} \left(\varphi (y, t_\eps - b \varepsilon^2)
        - \varphi (x_\eps, t_\eps - b \varepsilon^2)\right) \,{\rm d}y \right. \\[10pt]
        & \qquad\qquad\qquad \left. \displaystyle +\dashint_{B_{\varepsilon} (x_\eps)} \left(\varphi (x_\eps, t_\eps - b \varepsilon^2) - \varphi (x_\eps,t_\eps)\right) \,{\rm d}y\right\} + o(\eps^2) .
    \end{array}
\end{equation}
Now, we just observe that, since $\varphi$ is smooth, we have
$$
    \dashint_{B_{\varepsilon} (x_\eps)} \left(\varphi (y, t_\eps - b \varepsilon^2)- \varphi (x_\eps, t_\eps - b \varepsilon^2)\right) \,{\rm d}y
    = c(N) \eps^2 \Delta \varphi (x_0,t_0) + o (\eps^2)
$$
and
$$
    \dashint_{B_{\varepsilon} (x_\eps)} \left(\varphi (x_\eps, t_\eps - b \varepsilon^2) - \varphi (x_\eps,t_\eps)\right) \,{\rm d}y =
    - \eps^2 b \frac{\partial \varphi}{\partial t} (x_0,t_0) + o(\eps^2).
$$
We have used the well known formula, for $f\in \mathcal{C}^2(B_\rho)$, for some $\rho>0$,
$$
    \lim_{\eps\to0}\frac1{\eps^{2}}\dashint_{B_\eps}(f(x)-f(0))\, {\rm d}x=c(N)\Delta f(0),
$$
where $c(N)$ is the universal constant given by
$$
    c(N) = \frac12\dashint_{B_1} y_1^2 \,{\rm d}y = \frac{1}{2N(N+2)}.
$$

Going back to \eqref{q1890} we get
\begin{equation}\label{q18901}
    \begin{array}{rl}
        0 &\displaystyle\leq \inf_{b\in [C b^-, C b^+]}c(N) \eps^2 \Delta \varphi (x_0,t_0)- \eps^2 b \frac{\partial \varphi}{\partial t} (x_0,t_0) + o(\eps^2) \\[10pt]
        & \displaystyle =  c(N) \eps^2 \Delta \varphi (x_0,t_0) - \sup_{b\in [C b^-, C b^+]} \eps^2 b \frac{\partial \varphi}{\partial t} (x_0,t_0) + o(\eps^2).
    \end{array}
\end{equation}
Dividing by $\eps^2$ and taking $\eps\to 0$ we arrive to
\begin{equation} \label{q18901a}
    \begin{array}{l}
        \displaystyle 0 \leq c(N)  \Delta \varphi (x_0,t_0) - \sup_{b\in [C b^-,C b^+]}  b \frac{\partial \varphi}{\partial t} (x_0,t_0) ,
    \end{array}
\end{equation}
that is, taking $b/C$ instead of $b$ when computing the supremum, we get
$$
    \displaystyle b(\varphi_t) \varphi_t (x_0,t_0) = C \sup_{b\in [b^-, b^+]}  b \frac{\partial \varphi}{\partial t} (x_0,t_0) \leq c(N)  \Delta \varphi (x_0,t_0),
$$
as we wanted to show if we let $C=c(N)$.

To end the proof, we must check that $u$ is also a viscosity supersolution to the equation in \eqref{eq.main.intro}, and the proof is completely analogous. Here we take a smooth function $\varphi\colon\Omega\times[0,\infty)\to \mathbb{R}$ such that $u-\varphi$ has a strict minimum in $B_\delta(x_0)\times (t_0-\delta,t_0+\delta)\subseteq \Omega\times (0,T)$ for some $\delta>0$ at $(x_0,t_0)$. Then, arguing as before, using the uniform convergence $u^\eps\to u$, we get that
$$
    u^\eps(y,s)-\varphi (y,s) \geq u^\eps (x_\eps,t_\eps)-\varphi(x_\eps,t_\eps) -\eps^3,
$$
with, as before,
$$
    (x_\eps,t_\eps) \to (x_0,t_0) \quad \mbox{ as } \eps \to 0.
$$
Then
$$
    u^\eps(y,s)-u^\eps (x_\eps,t_\eps) \geq\varphi (y,s)-\varphi(x_\eps,t_\eps) -o(\eps^2)
$$
and by the DPP at $(x_\eps,t_\eps)$,
\begin{equation}\label{q18-}
    \displaystyle 0= \inf_{b\in [C b^-, C b^+]}
    \dashint_{B_{\varepsilon} (x_\eps)} \left(u^\varepsilon (y, t_\eps - b \varepsilon^2)- u^\varepsilon (x_\eps,t_\eps)\right) \,{\rm d}y,
\end{equation}
we obtain
\begin{equation}\label{q189-}
    \displaystyle 0 \geq \inf_{b\in [C b^-, C b^+]}
    \dashint_{B_{\varepsilon} (x_\eps)} \left(\varphi (y, t_\eps - b \varepsilon^2)- \varphi (x_\eps,t_\eps)\right) \,{\rm d}y -o(\eps^2).
\end{equation}
That is,
    \begin{equation} \label{q1890-}
        \begin{array}{rl}
        0 &\displaystyle\geq \inf_{b\in [C b^-, C b^+]}\dashint_{B_{\varepsilon} (x_\eps)} \left(\varphi (y, t_\eps - b \varepsilon^2)
        - \varphi (x_\eps, t_\eps - b \varepsilon^2)\right) \,{\rm d}y \\[10pt]
        & \qquad \qquad \qquad \displaystyle +\dashint_{B_{\varepsilon} (x_\eps)} \left(\varphi (x_\eps, t_\eps - b \varepsilon^2) - \varphi (x_\eps,t_\eps)\right) \,{\rm d}y - o(\eps^2).
    \end{array}
\end{equation}
Now using that $\varphi$ is smooth, with the same computations used before we deduce
\begin{equation} \label{q18901-}
    \begin{array}{rl}
        0&\displaystyle \geq \inf_{b\in [C b^-, C b^+]}c(N) \eps^2 \Delta \varphi (x_0,t_0)- \eps^2 b \frac{\partial \varphi}{\partial t} (x_0,t_0) + o(\eps^2) \\[10pt]
        &\displaystyle=c(N)\eps^2\Delta\varphi(x_0,t_0)-\sup_{b\in [b^-,b^+]} \eps^2 b \frac{\partial \varphi}{\partial t} (x_0,t_0) + o(\eps^2).
    \end{array}
\end{equation}
Dividing by $\eps^2$ and taking $\eps\to 0$ we arrive to
\begin{equation} \label{q18901a-}
    \begin{array}{l}
        \displaystyle 0 \geq c(N)  \Delta \varphi (x_0,t_0) - \sup_{b\in [C b^-, C b^+]}  b \frac{\partial \varphi}{\partial t} (x_0,t_0) ,
    \end{array}
\end{equation}
that is,
$$
    \displaystyle C b(\varphi_t) \varphi_t (x_0,t_0) = C \sup_{b\in [b^-,b^+]}  b \frac{\partial \varphi}{\partial t} (x_0,t_0) \geq c(N)  \Delta \varphi (x_0,t_0),
$$
which shows that $u$ is a viscosity supersolution to the equation.

Hence we have obtained that any possible limit along subsequences $\varepsilon_j\to 0$ of $u^\varepsilon$ solves Problem~\eqref{eq.main.intro}--\eqref{eq.b.intro}. From the uniqueness of solutions to Problem~\eqref{eq.main.intro}--\eqref{eq.b.intro}
we conclude the uniform convergence
$$
    \lim\limits_{\varepsilon\to0}u^\varepsilon=u,
$$
which finishes the proof.
\end{proof}

%%%%%%%%%%%%%%%%%%%%%%%%%%%%%%%%%%%
\section*{Acknowledgements}

This research has been supported by MICIU/AEI/10.13039/501100011033, through grant CEX2023-001347-S.

\textsc{A.\,de Pablo} and \textsc{F.\,Quir\'os} were also supported by grants PID2023-146931NB-I00, RED2022-134784-T, and RED2024-153842-T, all of them funded by MICIU/AEI/10.13039/501100011033.

\textsc{J.\,D.\,Rossi} was also supported by CONICET grants PIP GI No 11220150100036CO, PICT-2018-03183 and UBACyT grant 20020160100155BA, Argentina.

%%%%%%%%%%%%%%%%%%%%%%%%%%%%%%%%%%%%%%%%%%%%%%%%%%%%%%%%%%%%%%
%---------------------------------
% BIBLIOGRAFIA
%---------------------------------

\bibliographystyle{plain}

\end{document}